\documentclass{amsart}

\usepackage{amsmath,amssymb,amsthm,mathtools,mathrsfs}
\usepackage{array,booktabs,tabularx}
\usepackage{enumitem}
\usepackage{aliascnt}
\usepackage[hidelinks]{hyperref}
\usepackage[nameinlink,noabbrev]{cleveref}

\newtheorem{theorem}{Theorem}[section]
\newtheorem*{theorem*}{Theorem}
\newtheorem*{maintheoremA}{Main Theorem A}
\newtheorem*{maintheoremB}{Main Theorem B}
\newtheorem*{maintheoremC}{Main Theorem C}

\newaliascnt{proposition}{theorem}
\newtheorem{proposition}[proposition]{Proposition}
\aliascntresetthe{proposition}
\newaliascnt{corollary}{theorem}
\newtheorem{corollary}[corollary]{Corollary}
\aliascntresetthe{corollary}
\newaliascnt{lemma}{theorem}
\newtheorem{lemma}[lemma]{Lemma}
\aliascntresetthe{lemma}

\theoremstyle{definition}
\newaliascnt{definition}{theorem}
\newtheorem{definition}[definition]{Definition}
\aliascntresetthe{definition}
\newaliascnt{example}{theorem}
\newtheorem{example}[example]{Example}
\aliascntresetthe{example}
\newaliascnt{assumption}{theorem}

\theoremstyle{remark}
\newaliascnt{remark}{theorem}
\newtheorem{remark}[remark]{Remark}
\aliascntresetthe{remark}

\numberwithin{equation}{section}
\allowdisplaybreaks
\crefname{section}{Section}{Sections}
\Crefname{section}{Section}{Sections}
\crefname{subsection}{Section}{Sections}
\Crefname{subsection}{Section}{Sections}
\crefname{subsubsection}{Section}{Sections}
\Crefname{subsubsection}{Section}{Sections}
\crefname{equation}{equation}{equations}
\Crefname{equation}{Equation}{Equations}
\crefname{definition}{Definition}{Definitions}
\Crefname{definition}{Definition}{Definitions}
\crefname{theorem}{Theorem}{Theorems}
\Crefname{theorem}{Theorem}{Theorems}
\crefname{proposition}{Proposition}{Propositions}
\Crefname{proposition}{Proposition}{Propositions}
\crefname{lemma}{Lemma}{Lemmas}
\Crefname{lemma}{Lemma}{Lemmas}
\crefname{corollary}{Corollary}{Corollaries}
\Crefname{corollary}{Corollary}{Corollaries}
\crefname{remark}{Remark}{Remarks}
\Crefname{remark}{Remark}{Remarks}
\crefname{example}{Example}{Examples}
\Crefname{example}{Example}{Examples}
\crefname{assumption}{Assumption}{Assumptions}
\Crefname{assumption}{Assumption}{Assumptions}

\DeclareMathOperator{\Lie}{Lie}
\DeclareMathOperator{\hy}{hy}
\DeclareMathOperator{\ind}{ind}
\DeclareMathOperator{\ch}{ch}
\DeclareMathOperator{\ad}{ad}
\DeclareMathOperator{\Ad}{Ad}
\DeclareMathOperator{\soc}{soc}
\DeclareMathOperator{\Hom}{Hom}
\DeclareMathOperator{\sRep}{sRep}
\DeclareMathOperator{\Rep}{Rep}
\DeclareMathOperator{\Ber}{Ber}
\DeclareMathOperator{\Ext}{Ext}
\DeclareMathOperator{\St}{St}
\DeclareMathOperator{\str}{str}
\DeclareMathOperator{\rank}{rank}
\DeclareMathOperator{\Ker}{Ker}
\newcommand{\uHom}{\underline{\operatorname{Hom}}}
\newcommand{\G}{\mathbf G}
\newcommand{\Gev}{\mathbf G_{\mathrm{ev}}}
\newcommand{\GZ}{\mathbf G_{\mathbb Z}}
\newcommand{\GZev}{(\mathbf G_{\mathbb Z})_{\mathrm{ev}}}
\newcommand{\GLs}{\mathbf{GL}}
\newcommand{\SLs}{\mathbf{SL}}
\newcommand{\PGLs}{\mathbf{PGL}}
\newcommand{\PSLs}{\mathbf{PSL}}

\newcommand{\SpOs}{\mathbf{SpO}}
\newcommand{\Bplus}{\mathbf B^{+}}
\newcommand{\Bminus}{\mathbf B^{-}}
\newcommand{\Bpm}{\mathbf B^{\pm}}
\newcommand{\Pplus}{\mathbf P^{+}}
\newcommand{\Pminus}{\mathbf P^{-}}
\newcommand{\Ppm}{\mathbf P^{\pm}}
\newcommand{\Voddplus}{\mathbf V^{+}}
\newcommand{\Voddminus}{\mathbf V^{-}}
\newcommand{\Vpm}{\mathbf V^{\pm}}
\newcommand{\Dzero}{\Delta_{\bar 0}}
\newcommand{\Done}{\Delta_{\bar 1}}
\newcommand{\Dzerop}{\Delta_{\bar 0}^{+}}
\newcommand{\Donep}{\Delta_{\bar 1}^{+}}
\newcommand{\Donepi}[1]{\Delta_{\bar 1,[#1]}^{+}}
\newcommand{\eps}{\varepsilon}
\newcommand{\del}{\delta}
\newcommand{\gmone}{\mathfrak g_{-1}}
\newcommand{\gpone}{\mathfrak g_{1}}
\newcommand{\id}{\mathrm{id}}
\newcommand{\KacG}{\mathcal K_{\G}}
\newcommand{\coKacG}{\mathcal K^{\vee}_{\G}}
\newcommand{\Gr}{\G_r}
\newcommand{\KacGr}{\mathcal K_{\Gr}}
\newcommand{\JGr}{\mathcal J_{\Gr}}
\newcommand{\Dtyp}{\mathscr D_{\G}}
\newcommand{\Lambdaplus}{\Lambda_{+}}
\newcommand{\Lambdaminus}{\Lambda_{-}}
\newcommand{\WeylC}{\mathsf W}
\newcommand{\pr}{\mathsf{pr}}
\newcommand{\can}{\mathsf{can}}
\newcommand{\ba}{\mathbf a}
\newcommand{\bb}{\mathbf b}
\newcommand{\Gm}{\mathrm{G}_{\mathrm m}}
\newcommand{\GLord}{\mathrm{GL}}
\newcommand{\SLord}{\mathrm{SL}}
\newcommand{\Oord}{\mathrm O}
\newcommand{\SOord}{\mathrm{SO}}
\newcommand{\Spord}{\mathrm{Sp}}

\newcommand{\longhookrightarrow}{\lhook\joinrel\longrightarrow}
\newcommand{\Tq}{\overline T}

\begin{document}

\title[Typical representations]{Typical representations of admissible type~$\mathrm I$ algebraic supergroups}
\author[T.~Shibata]{Taiki Shibata}
\address{Department of Applied Mathematics, Okayama University of Science, 1-1 Ridai-cho, Kita-ku, Okayama-shi, Okayama 700-0005, JAPAN}
\email{shibata@ous.ac.jp}
\date{\today}

\begin{abstract}
We study typical representations of an admissible type~$\mathrm I$ algebraic supergroup $\G$.
We prove that, for a dominant weight $\lambda$, the induced supermodule $H^0_{\G}(\lambda)$ is simple if and only if $\lambda$ is typical and the induced module $H^0_{\Gev}(\lambda)$ for the even subgroup $\Gev$ is simple.
We establish equivalences between typical blocks for $\G$ and blocks for $\Gev$, and give a projectivity criterion for Kac modules over every Frobenius kernel.
For connected $\SpOs(2\ell|2)$ over an algebraically closed field of odd characteristic, we classify all simple projective supermodules over Frobenius kernels.
\end{abstract}

\maketitle

\tableofcontents

\section{Introduction}\label{sec:intro}
A basic question in the representation theory of algebraic supergroups is how their representations are related to those of their even subgroups.
We study this question for typical representations in a setting that includes several families of supergroups.

Kac showed that typical simple representations of basic classical Lie superalgebras over $\mathbb C$ can be described in terms of the even part \cite{Kac1977,Kac1978}.
He obtained explicit character formulas and, in type $\mathrm I$, constructed these representations by induction from simple modules for the even part.
For a fixed strongly typical central character, Gorelik established an equivalence between suitable categories of representations of a basic classical Lie superalgebra over $\mathbb C$ and of its even part \cite{Gorelik2002}; see also \cite{PenkovSerganova1992}.
An analogous reduction to the even part is available in positive characteristic.
For type $\mathrm I$ basic classical Lie superalgebras, Zhao proved that induction from the even part gives an equivalence between the categories of modules with typical composition factors for the reduced enveloping algebras of the Lie superalgebra and its even part \cite{Zhao2011}.

Related typicality criteria occur for quantum supergroups at roots of unity.
For the unrolled quantum supergroup of $\mathfrak{osp}(2|2\ell)$, Porter and Young characterized the simplicity of modules induced from simple modules for the even part by a quantum analogue of the typicality condition \cite{PorterYoung2026}.

For algebraic supergroups, the objects studied in this paper, Marko extended Kac's simplicity criterion in the general linear case to positive characteristic \cite{Marko2018}.
More precisely, for $\G=\GLs(m|n)$ and a dominant weight $\lambda$, he proved that the induced supermodule $H^0_{\G}(\lambda)$ is simple if and only if $\lambda$ is typical and the induced module $H^0_{\G_{\mathrm{ev}}}(\lambda)$ for $\G_{\mathrm{ev}}=\GLord_m\times\GLord_n$ is simple.
In characteristic zero, the condition on $H^0_{\G_{\mathrm{ev}}}(\lambda)$ is automatic, and one recovers Kac's criterion.
For $\GLs(m|n)$ over an algebraically closed field of characteristic $p>2$, Li and Shu used odd reflections to construct Jantzen filtrations of Weyl modules with highest weights that are typical in Kac's sense and establish a sum formula for the characters of the filtration terms.
As an application of their methods, they showed that a simple supermodule can be realized as a Kac module if and only if its highest weight is $p$-typical \cite{LiShu2023}.
Their $p$-typical weights are precisely the typical weights in our terminology.

Together, these results illustrate the fundamental role of typicality in the representation theory of Lie superalgebras, quantum supergroups, and algebraic supergroups.

In this paper, we extend Marko's simplicity criterion to a class of algebraic supergroups.
We also prove an analogue of Zhao's equivalence for their finite-dimensional modules with typical composition factors.

To treat $\GLs(m|n)$, connected ortho-symplectic supergroups $\SpOs(2\ell|2)$, and their central quotients in a common setting, we introduce the notion of an \emph{admissible type~$\mathrm I$ algebraic supergroup}.
Roughly speaking, such a supergroup $\G$ has a split connected reductive even subgroup $\Gev$ and a $\Gev$-stable $\mathbb Z$-grading
\[
\mathfrak{g}:=\Lie(\G)=\gmone\oplus\Lie(\Gev)\oplus\gpone,
\]
where $\gmone\oplus\gpone$ is the odd part of $\mathfrak g$.
We also require that the positive odd roots $\Donep$ can be ordered so that reversing their signs one at a time gives a set of positive roots at every step, while the positive even roots remain fixed.
The precise assumptions are given in \Cref{def:admissible_type_I}.
In particular, the $\mathbb Z$-grading of $\mathfrak{g}$ gives two parabolic subsupergroups of $\G$ containing opposite Borel subsupergroups.

Let $\G$ be an admissible type~$\mathrm I$ algebraic supergroup over a field $\Bbbk$ of characteristic different from $2$.
Fix a split maximal torus $T\subseteq\Gev$ and sets $\Dzerop$ and $\Donep$ of positive even and odd roots such that $\mathfrak g_{\pm1}$ is the sum of the root spaces of weights $\pm\beta$ for $\beta\in\Donep$.
These odd root spaces are one-dimensional.
Let $\Bpm$ be the corresponding Borel subsupergroups.
Write $\mathsf X(T)$ for the character group of $T$.
For $\lambda\in\mathsf X(T)$, we consider the induced $\G$-supermodule and $\Gev$-module
\[
H^0_{\G}(\lambda):=\ind_{\Bminus}^{\G}(\Bbbk_\lambda),
\qquad
H^0_{\Gev}(\lambda):=\ind_{\Bminus_{\mathrm{ev}}}^{\Gev}(\Bbbk_\lambda),
\]
respectively.
Here $\Bbbk_\lambda$ is the purely even one-dimensional $T$-supermodule of weight $\lambda$.
We regard $\Bbbk_\lambda$ as a $\Bminus$-supermodule through the quotient $\Bminus\to T$.
For each $\beta\in\Donep$, choose nonzero root vectors $X_\beta$ and $Y_\beta$ of weights $\beta$ and $-\beta$.
Put $H_\beta:=[X_\beta,Y_\beta]$ and $\rho:=\frac12\sum_{\alpha\in\Dzerop}\alpha-\frac12\sum_{\beta\in\Donep}\beta$.
A weight $\lambda\in\mathsf X(T)$ is \emph{typical} if its \emph{typicality scalar}
\[
\Dtyp(\lambda):=\prod_{\beta\in\Donep}(\lambda+\rho)(H_\beta)\in\Bbbk
\]
is nonzero, and \emph{atypical} otherwise (see \Cref{def:typical}).
We show that a suitable product of positive and negative odd root vectors acts on a highest weight vector of weight $\lambda$ by this scalar (\Cref{thm:odd_product}).
This computation gives the following extension of Marko's simplicity criterion from $\GLs(m|n)$ to admissible type~$\mathrm I$ algebraic supergroups.

\begin{maintheoremA}[\Cref{thm:main_irred}]
For every dominant weight $\lambda$,
\[
H^0_{\G}(\lambda)\text{ is simple}
\quad\Longleftrightarrow\quad
H^0_{\Gev}(\lambda)\text{ is simple and }\lambda\text{ is typical}.
\]
\end{maintheoremA}

For every dominant weight $\lambda$, the induced supermodule $H^0_{\G}(\lambda)$ has a simple socle \cite{Shibata2020}, which we denote by $L_{\G}(\lambda)$.
We write $L_{\Gev}(\lambda)$ for the simple socle of $H^0_{\Gev}(\lambda)$.
By the $\Gev$-stable $\mathbb Z$-grading of $\mathfrak g$, there are parabolic subsupergroups $\Ppm\cong\Gev\ltimes\Vpm$ of $\G$ containing $\Bpm$, where $\Vpm$ have trivial even parts with $\Lie(\Vpm)=\mathfrak g_{\pm1}$.
Regard a finite-dimensional $\Gev$-supermodule $M$ as a $\Pplus$-supermodule through the quotient $\Pplus\to\Gev$.
Write $\hy(\mathbf H)$ for the hyperalgebra, or algebra of distributions, of an algebraic supergroup $\mathbf H$ (see \Cref{subsec:opposite_distinguished_parabolics}).
We define the \emph{Kac module}
\[
\KacG(M):=\hy(\G)\otimes_{\hy(\Pplus)}M
\]
of $M$.
Using the $T$-action induced by the adjoint action on $\hy(\G)$ and the given action on $M$, we regard $\KacG(M)$ as a finite-dimensional $\G$-supermodule (\Cref{lem:Kac_properties}).
For a dominant weight $\lambda$, we prove that $\KacG(L_{\Gev}(\lambda))$ is simple if and only if $\lambda$ is typical (\Cref{thm:Kac_simple}).
In this case, we have (\Cref{cor:typical_simple_kac_realization})
\[
L_{\G}(\lambda)\cong\KacG(L_{\Gev}(\lambda)).
\]
In characteristic zero, $H^0_{\Gev}(\lambda)$ is simple for every dominant weight $\lambda$, so Main Theorem~A says that $H^0_{\G}(\lambda)$ is simple if and only if $\lambda$ is typical.
For $\GLs(m|n)$ (type~$\mathtt A$) and $\SpOs(2\ell|2)$ (type~$\mathtt C$) over $\mathbb C$, this gives Kac's simplicity criterion, and the displayed isomorphism recovers his construction of typical simple modules.
In characteristic zero, Chen and Mazorchuk proved a corresponding simplicity criterion for Kac modules over certain classical type~$\mathrm I$ Lie superalgebras \cite{ChenMazorchuk2021}.

For an algebraic supergroup $\mathbf H$, including a purely even one, let $\sRep(\mathbf H)$ denote the category of finite-dimensional left $\mathbf H$-supermodules with parity-preserving morphisms.
The assignment $M\mapsto\KacG(M)$ defines the \emph{Kac functor}
\[
\KacG:\sRep(\Gev)\longrightarrow\sRep(\G).
\]
Let $\mathscr E\subseteq\sRep(\Gev)$ and $\mathscr C^{\mathrm{typ}}\subseteq\sRep(\G)$ be the full subcategories consisting of supermodules whose simple composition factors have typical highest weights.
The next theorem reduces the study of these $\G$-supermodules to representations of $\Gev$.
Write $W^{\gpone}$ for the subspace of a $\G$-supermodule $W$ annihilated by $\gpone$.

\begin{samepage}
\begin{maintheoremB}[\Cref{thm:Serre_equiv}]
\leavevmode
\begin{enumerate}[label=\textup{(\roman*)}]
\item Every finite-dimensional $\G$-supermodule $W$ has a natural decomposition
\[
W=W^{\mathrm{typ}}\oplus W^{\mathrm{atyp}},
\]
where $W^{\mathrm{typ}}\in\mathscr C^{\mathrm{typ}}$ and every simple composition factor of $W^{\mathrm{atyp}}$ has atypical highest weight.
\item The functors
\[
\KacG:\mathscr E\longrightarrow\mathscr C^{\mathrm{typ}},
\qquad
(-)^{\gpone}:\mathscr C^{\mathrm{typ}}\longrightarrow\mathscr E
\]
are inverse equivalences of abelian categories.
\end{enumerate}
\end{maintheoremB}
\end{samepage}

For $\GLs(m|n)$, Marko and Zubkov showed that the sets of highest weights in typical blocks coincide with those in the corresponding blocks of the even subgroup \cite{MarkoZubkovGLBlocks}.
More generally, for every admissible type~$\mathrm I$ supergroup $\G$, we prove that each block containing a typical simple supermodule is equivalent to an ordinary $\Gev$-block containing a typical simple module (\Cref{thm:typical_blocks}).
This equivalence also identifies extensions. For typical dominant weights $\lambda,\mu$ and every $n\ge0$, we have
\[
\Ext^n_{\G}(L_{\G}(\lambda),L_{\G}(\mu))
\cong
\Ext^n_{\Gev}(L_{\Gev}(\lambda),L_{\Gev}(\mu))
\qquad\text{(\Cref{cor:serre_consequences}).}
\]

We also determine which Kac modules are projective over Frobenius kernels.
Assume now that $\Bbbk$ is perfect of odd characteristic, and let $\G_r$ and $\G_{\mathrm{ev},r}$ be the $r$-th Frobenius kernels, where $r\ge1$.
For a finite-dimensional $\G_{\mathrm{ev},r}$-supermodule $M$, set
\[
\KacGr(M):=\hy(\G_r)\otimes_{\hy(\Pplus_r)}M,
\qquad \Pplus_r:=\Pplus\cap\G_r,
\]
where $M$ is regarded as a $\Pplus_r$-supermodule through the quotient $\Pplus_r\to\G_{\mathrm{ev},r}$.

To construct the endomorphism in the projectivity criterion, we use the multiplication isomorphism
\[
\bigwedge(\gmone)\otimes_{\Bbbk}\hy(\Gev)\otimes_{\Bbbk}\bigwedge(\gpone)
\xrightarrow{\ \cong\ }\hy(\G)
\]
of supercoalgebras \cite{Masuoka2012}.
Let $\pr:\hy(\G)\to\hy(\Gev)$ be the projection onto degree zero in both exterior factors.
Choose an ordering $\beta_1\prec\cdots\prec\beta_N$ of $\Donep$ as required in \Cref{def:admissible_type_I}, and set
\begin{equation}\label{eq:intro_EF}
E:=X_{\beta_1}\cdots X_{\beta_N},\qquad
F:=Y_{\beta_N}\cdots Y_{\beta_1}
\end{equation}
in $\hy(\G)$.
Put $\Theta:=\pr(EF)$.
The element $\Theta$ is central in $\hy(\Gev)$ and acts on $L_{\Gev}(\lambda)$ by the scalar $\Dtyp(\lambda)$ (\Cref{lem:centralTheta}).
Since $\Theta\in\hy(\G_{\mathrm{ev},1})\subseteq\hy(\G_{\mathrm{ev},r})$, its action on $M$ defines an endomorphism $\Theta_M:M\to M$.

\begin{samepage}
\begin{maintheoremC}[\Cref{thm:general_projectivity}]
\leavevmode
\begin{enumerate}[label=\textup{(\roman*)}]
\item For a finite-dimensional $\G_{\mathrm{ev},r}$-supermodule $M$, the module $\KacGr(M)$ is projective if and only if $M$ is projective and $\Theta_M$ is invertible.
\item The functor $\KacGr$ gives a bijection between the isomorphism classes of simple projective $\G_{\mathrm{ev},r}$-supermodules with invertible $\Theta$-action and those of simple projective $\G_r$-supermodules.
\end{enumerate}
\end{maintheoremC}
\end{samepage}

For every typical dominant weight $\lambda$, Main Theorem~C gives (\Cref{cor:H0_projectivity})
\[
L_{\G}(\lambda)|_{\G_r}\text{ is projective}
\quad\Longleftrightarrow\quad
L_{\Gev}(\lambda)|_{\G_{\mathrm{ev},r}}\text{ is projective}.
\]
Thus we can first find the simple $\Gev$-modules that are projective over $\G_{\mathrm{ev},r}$.
Applying $\KacG$ to those of typical highest weight then gives simple $\G$-supermodules that are projective over $\G_r$.

For $\GLs(m|n)$, Zubkov proved that the $r$-th Steinberg supermodules are projective over $\Gr$ for every $r\ge1$ \cite[Lemma~10.7]{ZubkovGoodWeyl}.
Main Theorem~C recovers this result and gives a projectivity criterion for $\KacGr(M)$ for an arbitrary finite-dimensional $\G_{\mathrm{ev},r}$-supermodule $M$.
For $\GLs(m|n)$ and $\SpOs(2\ell|2)$ over an algebraically closed field, the bijection in Main Theorem~C\textup{(ii)} for $r=1$ can also be deduced from the results of Zhao and Yao \cite{Zhao2011,Yao2014} (see the discussion following \Cref{thm:general_projectivity}).

Applying Main Theorems~B and~C to connected $\SpOs(2\ell|2)$ and its quotient by the scalar subgroup of order $2$ over an algebraically closed field of odd characteristic, we obtain explicit descriptions of the typical blocks and a classification of all simple projective supermodules over the Frobenius kernels (\Cref{thm:typeC_blocks,thm:typeC_steinberg}).

Porter and Young construct a simple projective Kac module induced from an even module of Steinberg highest weight in the quantum setting \cite{PorterYoung2026}.
Our construction gives an analogue of their result for Frobenius kernels of $\SpOs(2\ell|2)$.

\subsection*{Organization of the paper}

In \Cref{sec:setup}, we introduce opposite distinguished parabolics and admissible orderings of the odd roots.
We give a criterion for the existence of these parabolics in terms of the Lie superalgebra and establish the PBW decompositions used to define the Kac functors.
We verify the assumptions for $\GLs(m|n)$, $\SLs(m|n)$ with $m\ne n$, and their scalar central quotients (type~$\mathtt A$), and for connected $\SpOs(2\ell|2)$ and its scalar quotient of order $2$ (type~$\mathtt C$).
The construction of the quotients over $\mathbb Z$ is included.

In \Cref{sec:typicality}, we identify induced supermodules with the values of the dual Kac functor.
An admissible ordering allows us to compute the action of $EF$ on a highest weight vector.
We use this computation and a natural map between the two Kac functors to prove the simplicity criterion in Main Theorem~A.
The section also contains the character formula for typical simple supermodules and an example outside types~$\mathtt A$ and~$\mathtt C$.

In \Cref{sec:general-applications}, we first introduce the categories of supermodules with typical composition factors and prove Main Theorem~B.
We then describe the typical blocks and identify their extension groups with those for the even subgroup.
After introducing Frobenius kernels, we prove the projectivity criterion for Kac modules and the bijection between simple projective supermodules in Main Theorem~C.

In \Cref{sec:examples}, we calculate the typicality scalars in types~$\mathtt A$ and~$\mathtt C$.
For connected $\SpOs(2\ell|2)$, the known blocks of the symplectic group give explicit descriptions of the typical blocks, and an example for $\SpOs(2|2)$ gives a nonsplit extension between typical simple supermodules.
Finally, the Steinberg module of the even symplectic factor gives the classification of simple projective supermodules over all Frobenius kernels.
We realize these simple projective supermodules, up to parity shift, as restrictions of induced $\SpOs(2\ell|2)$-supermodules and determine which of the latter descend to the central quotient.

\subsection*{Acknowledgments}

The author was supported by JSPS KAKENHI Grant Number JP26K06774.
The author acknowledges the assistance of ChatGPT (GPT-6) in developing some of the mathematical arguments and improving the exposition.
The author takes full responsibility for the mathematical content and the final text.

\section{Distinguished parabolics and admissible orderings}\label{sec:setup}
Throughout this paper, we work over a field $\Bbbk$ of characteristic different from $2$.
We write $\otimes$ for the tensor product over $\Bbbk$.
All superspaces are $\mathbb Z_2$-graded, and homomorphisms of supermodules preserve parity.

\subsection{Lie superalgebras and roots}\label{subsec:setup_assumptions}
The supergroups considered here are obtained from split quasireductive supergroups over $\mathbb Z$.
More precisely, let $\GZ$ be such a supergroup, with even part $\GZev$, in the sense of \cite[Definition~3.1]{Shibata2020}.
See also \cite[Remark~1.6]{Hayashi2025} for the terminology.
Fix a split maximal torus $T_{\mathbb Z}\subseteq\GZev$, and put
\[
\G:=\GZ\times_{\mathbb Z}\Bbbk,\qquad
\Gev:=\GZev\times_{\mathbb Z}\Bbbk,\qquad
T:=T_{\mathbb Z}\times_{\mathbb Z}\Bbbk.
\]
Thus $\Gev$ is a split connected reductive group with split maximal torus $T$.
Set
\[
\mathfrak g:=\Lie(\G)=\mathfrak g_{\bar0}\oplus\mathfrak g_{\bar1},\qquad
\mathfrak h:=\Lie(T),
\]
so that $\mathfrak g_{\bar0}=\Lie(\Gev)$.
For an algebraic group $H$, write $\mathsf X(H):=\Hom(H,\Gm)$, where $\Gm$ is the multiplicative group.
We identify $\mathsf X(T_{\mathbb Z})$ with $\mathsf X(T)$.

Under the adjoint $T$-action, write $\mathfrak g=\bigoplus_{\alpha\in\mathsf X(T)}\mathfrak g^\alpha$.
For $\epsilon\in\mathbb Z_2$, let $\Delta_\epsilon$ be the set of nonzero weights of $\mathfrak g_\epsilon$, and set $\Delta:=\Dzero\cup\Done$.
The set $\Dzero$ is the root system of $(\Gev,T)$.
A homomorphism $\Upsilon:\mathbb Z\Delta\to\mathbb R$ that is nonzero on every root determines the positive and negative roots
\[
\Delta^\pm_\epsilon(\Upsilon):=\{\gamma\in\Delta_\epsilon\mid\pm\Upsilon(\gamma)>0\},\qquad
\Delta^\pm(\Upsilon):=\Delta^\pm_{\bar0}(\Upsilon)\cup\Delta^\pm_{\bar1}(\Upsilon).
\]
Fix $\Upsilon$, and write $\Delta^\pm:=\Delta^\pm(\Upsilon)$, $\Dzerop:=\Delta^+_{\bar0}(\Upsilon)$, and $\Donep:=\Delta^+_{\bar1}(\Upsilon)$.
Put $\mathfrak g_{\pm1}:=\bigoplus_{\beta\in\Donep}\mathfrak g_{\bar1}^{\pm\beta}$.
We assume that
\begin{equation}\label{eq:paired_odd_roots}
\Done=\Donep\sqcup(-\Donep),\qquad
\mathfrak g_{\bar1}=\gpone\oplus\gmone,\qquad
\dim\mathfrak g_{\bar1}^{\pm\beta}=1\quad(\beta\in\Donep).
\end{equation}
In particular, $\mathfrak g^0=\mathfrak h$.
Let $\Bplus$ and $\Bminus$ be the Borel subsupergroups with Lie superalgebras $\mathfrak h\oplus\bigoplus_{\alpha\in\Delta^+}\mathfrak g^\alpha$ and $\mathfrak h\oplus\bigoplus_{\alpha\in\Delta^-}\mathfrak g^\alpha$, respectively, as constructed in \cite{Shibata2020}.

\subsection{Opposite distinguished parabolics}\label{subsec:opposite_distinguished_parabolics}
For a finite-dimensional vector space $W$, give the exterior algebra $\bigwedge(W^*)$ the Hopf superalgebra structure in which the elements of $W^*$ are odd and primitive.
The corresponding supergroup has trivial even part and Lie superalgebra $W$ with zero bracket.
We regard the $d$th exterior power $\bigwedge^d(W)$ as having parity $d$ modulo $2$.
\begin{definition}\label{def:opposite_distinguished_parabolics}
We say that $\G$ \emph{has opposite distinguished parabolics} with respect to $\Upsilon$ if there are closed subsupergroups $\Bpm\subseteq\Ppm\subseteq\G$ and $\Vpm\subseteq\Ppm$ such that
\[
\Ppm_{\mathrm{ev}}=\Gev,\qquad
(\Vpm)_{\mathrm{ev}}=\{1\},\qquad
\Lie(\Vpm)=\mathfrak g_{\pm1},\qquad
[\mathfrak g_{\pm1},\mathfrak g_{\pm1}]=0,
\]
$\Gev$ normalizes $\Vpm$, and multiplication induces isomorphisms
\[
\Gev\ltimes\Vpm\xrightarrow{\ \cong\ }\Ppm,\qquad
\Bpm_{\mathrm{ev}}\ltimes\Vpm\xrightarrow{\ \cong\ }\Bpm.
\]
The actions on $\Vpm$ are conjugation by $\Gev$ and $\Bpm_{\mathrm{ev}}$, respectively.
\end{definition}

\begin{proposition}\label{prop:opposite_parabolics_lie_criterion}
The supergroup $\G$ has opposite distinguished parabolics if and only if both $\mathfrak g_1$ and $\mathfrak g_{-1}$ are $\Gev$-stable and
\[
[\gpone,\gpone]=[\gmone,\gmone]=0.
\]
In this case, $\mathfrak g=\gmone\oplus\mathfrak g_{\bar0}\oplus\gpone$ is a $\Gev$-stable $\mathbb Z$-grading.
\end{proposition}

\begin{proof}
If $\G$ has opposite distinguished parabolics $\Ppm\cong\Gev\ltimes\Vpm$, the conjugation action of $\Gev$ preserves $\Vpm$.
Thus $\mathfrak g_{\pm1}$ are $\Gev$-stable and have zero bracket.

Conversely, suppose that $\mathfrak g_{\pm1}$ are $\Gev$-stable and $[\mathfrak g_{\pm1},\mathfrak g_{\pm1}]=0$.
Put $\mathfrak p^\pm:=\mathfrak g_{\bar0}\oplus\mathfrak g_{\pm1}$.
Then
\[
(\Gev,\mathfrak p^\pm),\qquad
(\{1\},\mathfrak g_{\pm1}),\qquad
(\Bpm_{\mathrm{ev}},\Lie(\Bpm_{\mathrm{ev}})\oplus\mathfrak g_{\pm1})
\]
are Harish--Chandra subpairs of $(\Gev,\mathfrak g)$.
By \cite[Proposition~2.6]{Shibata2020}, these subpairs give closed subsupergroups $\Ppm$, $\Vpm$, and $\Bpm$ of $\G$.
The last subsupergroups equal the chosen Borels, since their Harish--Chandra pairs are the same.
The pair $(\{1\},\mathfrak g_{\pm1})$ gives $\mathcal O(\Vpm)\cong\bigwedge(\mathfrak g_{\pm1}^*)$ as Hopf superalgebras.
The $\Gev$-stability of $\mathfrak g_{\pm1}$ implies that $\Gev$ normalizes $\Vpm$.
The multiplication maps $\Gev\ltimes\Vpm\to\Ppm$ and $\Bpm_{\mathrm{ev}}\ltimes\Vpm\to\Bpm$ induce isomorphisms of Harish--Chandra pairs, hence are isomorphisms by \cite[Theorem~29]{Masuoka2012}.
The inclusions of pairs also give $\Bpm\subseteq\Ppm$.
Finally, $[\mathfrak g_{\bar0},\mathfrak g_{\pm1}]\subseteq\mathfrak g_{\pm1}$ by stability, and $[\mathfrak g_1,\mathfrak g_{-1}]\subseteq\mathfrak g_{\bar0}$ by parity.
Together with the two zero brackets, these inclusions give the stated $\mathbb Z$-grading.
\end{proof}

For an algebraic supergroup $\mathbf H$, let
\[
\hy(\mathbf H):=\bigcup_{n\ge1}\left(\mathcal O(\mathbf H)/\mathfrak m_{\mathbf H}^n\right)^*,
\]
where $\mathfrak m_{\mathbf H}$ is the augmentation ideal of $\mathcal O(\mathbf H)$.
This is a cocommutative Hopf superalgebra with primitive elements $\Lie(\mathbf H)$.
Suppose that $\G$ has opposite distinguished parabolics $\Ppm\cong\Gev\ltimes\Vpm$.
The inclusions $\Vpm\subseteq\G$ identify $\Lambdaplus:=\bigwedge(\gpone)=\hy(\Voddplus)$ and $\Lambdaminus:=\bigwedge(\gmone)=\hy(\Voddminus)$ with exterior subalgebras of $\hy(\G)$.
Multiplication gives an isomorphism of supercoalgebras
\begin{equation}\label{eq:pbw_decomposition}
\Lambdaminus\otimes\hy(\Gev)\otimes\Lambdaplus\xrightarrow{\ \cong\ }\hy(\G)
\end{equation}
and isomorphisms of right $\hy(\Pplus)$-supermodules and left $\hy(\Pminus)$-supermodules, respectively,
\begin{equation}\label{eq:pbw_modules}
\Lambdaminus\otimes\hy(\Pplus)\xrightarrow{\ \cong\ }\hy(\G),
\qquad
\hy(\Pminus)\otimes\Lambdaplus\xrightarrow{\ \cong\ }\hy(\G).
\end{equation}
Indeed, choose an ordered basis of $\mathfrak g_{\bar1}$ in which the vectors in $\gmone$ precede those in $\gpone$.
By \cite[Proposition~12]{Masuoka2012}, multiplication identifies $\hy(\Gev)\otimes\Lambdaminus\otimes\Lambdaplus$ with $\hy(\G)$ as supercoalgebras.
The semidirect products $\Ppm=\Gev\ltimes\Vpm$ allow the factors $\hy(\Gev)$ and $\Lambdaminus$ to be interchanged and give \eqref{eq:pbw_decomposition} and \eqref{eq:pbw_modules}.
In particular, $\hy(\G)$ is finite free on the two sides in \eqref{eq:pbw_modules}.

\subsection{Admissible orderings of the odd roots}\label{subsec:oddflip_assumption}
To study simple supermodules, we shall apply odd root vectors successively to a highest weight vector.
For this purpose, we need to change the positive odd roots one at a time while keeping the positive even roots fixed.
\begin{definition}\label{def:admissible_ordering}
An ordering $\Donep=\{\beta_1\prec\cdots\prec\beta_N\}$, where $N:=|\Donep|$, is \emph{admissible} if, for every $0\le i\le N$, there is a homomorphism $\Upsilon_{[i]}:\mathbb Z\Delta\to\mathbb R$ nonzero on every root such that
\[
\Delta^+(\Upsilon_{[i]})=\Dzerop\cup\Donepi{i},\qquad
\Donepi{i}:=\{-\beta_1,\ldots,-\beta_i,\beta_{i+1},\ldots,\beta_N\},
\]
with $\Upsilon_{[0]}=\Upsilon$.
Thus $\Donepi{0}=\Donep$ and $\Donepi{N}=-\Donep$.
\end{definition}

The existence of an admissible ordering implies $\Dzero\cap\Done=\varnothing$.
Indeed, each odd root changes sign, whereas each even root keeps its sign.

\begin{remark}\label{rem:bad_order_GL21}
An arbitrary ordering of $\Donep$ need not be admissible.
For the diagonal torus of $\GLs(2|1)$, write $\eps_1,\eps_2,\del_1$ for the coordinate characters and put $\beta_i:=\eps_i-\del_1$ for $i=1,2$.
Choose $\eps_1-\eps_2$ as the positive even root and $\beta_1,\beta_2$ as the positive odd roots.
The equality $\beta_1=(\eps_1-\eps_2)+\beta_2$ shows that $\beta_1$ remains positive as long as $\eps_1-\eps_2$ and $\beta_2$ do.
Hence an admissible ordering must change $\beta_2$ before $\beta_1$.
\end{remark}

\begin{lemma}\label{lem:h1_and_abelian}
Suppose that $\Donep$ has an admissible ordering.
Then $\beta+\gamma\notin\Dzero$ for every $\beta,\gamma\in\Donep$.
In particular, $2\beta\notin\Dzero$ for $\beta\in\Donep$, and
\[
[\gpone,\gpone]=[\gmone,\gmone]=0.
\]
\end{lemma}

\begin{proof}
If $\alpha:=\beta+\gamma$ were an even root, it would be positive for $\Upsilon_{[0]}$ and negative for $\Upsilon_{[N]}$.
This contradicts the requirement that the positive even roots remain fixed.
Since $\Upsilon(\beta+\gamma)>0$, the sum is nonzero, and
\[
[\mathfrak g_{\bar1}^{\beta},\mathfrak g_{\bar1}^{\gamma}]
\subseteq\mathfrak g_{\bar0}^{\beta+\gamma}=0.
\]
Applying the same argument to $-\beta$ and $-\gamma$ gives $[\gmone,\gmone]=0$.
\end{proof}

\begin{corollary}\label{cor:opposite_parabolics_lie_chain}
Suppose that $\Donep$ has an admissible ordering.
Then $\G$ has opposite distinguished parabolics if and only if $\gpone$ and $\gmone$ are $\Gev$-stable.
\end{corollary}

\begin{proof}
This follows from \Cref{prop:opposite_parabolics_lie_criterion,lem:h1_and_abelian}.
\end{proof}

\begin{definition}\label{def:admissible_type_I}
We call $\G$ an \emph{admissible type~$\mathrm I$ algebraic supergroup} with respect to $(T,\Upsilon)$ if \eqref{eq:paired_odd_roots} holds, $\G$ has opposite distinguished parabolics, and $\Donep$ has an admissible ordering.
\end{definition}

\subsection{Examples}\label{subsec:setup_examples}
In this subsection, we give examples of admissible type~$\mathrm I$ algebraic supergroups.
We first prove the following proposition on fppf quotients by central subgroups.
\begin{proposition}\label{prop:central_quotient_general}
Let $C\subseteq T$ be a purely even central diagonalizable closed subgroup of $\G$.
If $\G$ has opposite distinguished parabolics $\Ppm\cong\Gev\ltimes\Vpm$ and $\Donep$ has an admissible ordering, then $\G/C$ has opposite distinguished parabolics and an admissible ordering with respect to $T/C$, with roots identified by the character map induced by $T\to T/C$.
\end{proposition}

\begin{proof}
Since $\Gev/C$ is affine, the quotient $\overline\G:=\G/C$ is representable and affine by \cite[Corollary~8.15]{MasuokaZubkov2011}, and its even part is $\Gev/C$.
Write $q:\G\to\overline\G$ for the quotient morphism and $q_{\mathrm{ev}}:\Gev\to\Gev/C$ for its restriction to the even subgroup.
Put $\overline T:=T/C$, and let $q_T:T\to\overline T$ be the induced quotient morphism.
By \cite[Proposition~22]{Masuoka2012}, choose a counit-preserving isomorphism of left $\mathcal O(\Gev)$-comodule superalgebras
\[
\mathcal O(\G)\cong\mathcal O(\Gev)\otimes\bigwedge(\mathfrak g_{\bar1}^*).
\]
Since $q$ is an fppf $C$-torsor, taking invariants for left translation by $C$ gives
\[
\mathcal O(\overline\G)=\mathcal O(\G)^C
\cong\mathcal O(\Gev/C)\otimes\bigwedge(\mathfrak g_{\bar1}^*).
\]
Under these identifications, the corresponding Hopf superalgebra map $\mathcal O(q):\mathcal O(\overline\G)\to\mathcal O(\G)$ is $\mathcal O(q_{\mathrm{ev}})\otimes\id$.
Thus $\Lie(q):\Lie(\G)\to\Lie(\overline\G)$ is an isomorphism on the odd parts.
The map $q_{\mathrm{ev}}$ restricts to an isomorphism on every even root subgroup, and centrality of $C$ makes every root trivial on $C$.
Restricting $\mathcal O(q_T):\mathcal O(\overline T)\to\mathcal O(T)$ to grouplike elements gives the injective character map
\[
\iota_T:\mathsf X(\overline T)\longhookrightarrow\mathsf X(T);\quad
\chi\longmapsto\chi\circ q_T.
\]
This map identifies the even and odd roots, and $\Lie(q)$ identifies their root spaces.

Since $C$ centralizes $\Vpm$, the conjugation action of $\Gev$ on $\Vpm$ descends to $\Gev/C$.
The homomorphisms
\[
\Ppm=\Gev\ltimes\Vpm\longrightarrow(\Gev/C)\ltimes\Vpm;\quad
(h,v)\longmapsto(q_{\mathrm{ev}}(h),v)
\]
are faithfully flat with kernel $C$.
The same argument applies to $\Bpm$, so
\[
\Ppm/C\cong(\Gev/C)\ltimes\Vpm,\qquad
\Bpm/C\cong(\Bpm_{\mathrm{ev}}/C)\ltimes\Vpm.
\]
The inclusion $\Ppm\subseteq\G$ induces the canonical Hopf superalgebra quotient map $\mathcal O(\G)\to\mathcal O(\Ppm)$, which is $C$-equivariant.
Taking $C$-invariants is exact because $C$ is diagonalizable, so $\mathcal O(\G)^C\to\mathcal O(\Ppm)^C$ is also surjective.
The invariant algebras are the coordinate algebras of the affine quotients by \cite[Theorem~7.1]{MasuokaZubkov2011}.
Hence $\Ppm/C\to\overline\G$ are closed immersions, and the same argument applies to $\Bpm/C$.
It follows that $\Ppm/C$ are opposite distinguished parabolics of $\overline\G$, containing $\Bpm/C$.
Composing the functionals $\Upsilon_{[i]}$ with $\iota_T$ on the root lattices gives an admissible ordering for $\overline\G$.
\end{proof}

\subsubsection{Type $\mathtt A$}\label{subsubsec:setup_typeA}
For $m,n\ge1$, let $\GLs(m|n)$ be the general linear supergroup, and put $\SLs(m|n):=\Ker(\Ber)$, where $\Ber:\GLs(m|n)\to\Gm$ is the Berezinian.
The matrix definition of $\GLs(m|n)$ and the Berezinian define split quasireductive supergroups over $\mathbb Z$ whose base changes to $\Bbbk$ are these supergroups.
Their even groups are $\GLord_m\times\GLord_n$ and $\{(A,D)\in\GLord_m\times\GLord_n\mid\det A=\det D\}$, respectively.
For $\GLs(m|n)$, choose the diagonal torus in $\GLord_m\times\GLord_n$, with coordinate characters $\eps_1,\ldots,\eps_m$ and $\del_1,\ldots,\del_n$.
For the upper triangular Borel in the standard homogeneous basis, choose
\[
\begin{aligned}
\Dzerop&=\{\eps_i-\eps_j\mid1\le i<j\le m\}\sqcup\{\del_i-\del_j\mid1\le i<j\le n\},\\
\Donep&=\{\eps_i-\del_j\mid1\le i\le m,\ 1\le j\le n\}.
\end{aligned}
\]
Let $\Pplus$ and $\Pminus$ be the block upper and lower parabolics, respectively.
Their purely odd subgroups $\Voddplus$ and $\Voddminus$ consist of matrices with identity diagonal blocks and an arbitrary odd upper-right or lower-left block, respectively.

\begin{proposition}\label{prop:GL_SL}
The supergroup $\GLs(m|n)$ has opposite distinguished parabolics and an admissible ordering for all $m,n\ge1$.
The same assertions hold for $\SLs(m|n)$ when $m\ne n$.
\end{proposition}

\begin{proof}
The upper-right and lower-left odd matrix subspaces are $\GLord_m\times\GLord_n$-stable and have one-dimensional weight spaces of weights $\Donep$ and $-\Donep$.
They satisfy \eqref{eq:paired_odd_roots} and $[\gpone,\gpone]=[\gmone,\gmone]=0$.
Thus \Cref{prop:opposite_parabolics_lie_criterion} gives opposite distinguished parabolics.
Order the positive odd roots by
\[
\eps_m-\del_1\prec\cdots\prec\eps_m-\del_n
\prec\eps_{m-1}-\del_1\prec\cdots\prec\eps_1-\del_n.
\]
For $1\le i\le m$ and $0\le j\le n$, put $s:=(m-i)n+j$.
For $s>0$, assign strictly decreasing real values to the characters in the list
\[
\eps_1,\ldots,\eps_{i-1},\del_1,\ldots,\del_j,\eps_i,\del_{j+1},\ldots,\del_n,\eps_{i+1},\ldots,\eps_m.
\]
The resulting homomorphism is positive on $\eps_a-\eps_b$ and $\del_a-\del_b$ for $a<b$.
Among the roots $\eps_a-\del_b$, it is negative exactly when $a>i$ or when $a=i$ and $b\le j$.
Its restriction to $\mathbb Z\Delta$ therefore gives $\Upsilon_{[s]}$ with positive roots $\Dzerop\cup\Donepi{s}$.
For $s=0$, take $\Upsilon_{[0]}=\Upsilon$.

For $m\ne n$, put $T':=T\cap\SLs(m|n)$.
Since $\mathsf X(T')=\mathsf X(T)/\mathbb Z(\sum_i\eps_i-\sum_j\del_j)$ and every root has total coefficient zero, restriction is injective on the root lattice.
The result follows by intersecting the two parabolics with $\SLs(m|n)$ and restricting the above ordering.
\end{proof}

The scalar center $D_{m|n}:=\{tI_{m+n}\mid t\in\Gm\}$ gives the fppf quotients
\[
\PGLs(m|n):=\GLs(m|n)/D_{m|n},\qquad
\PSLs(m|n):=\SLs(m|n)/C_{m|n},
\]
where $I_r$ denotes the $r\times r$ identity matrix and $C_{m|n}:=D_{m|n}\cap\SLs(m|n)$.
They are affine algebraic supergroups by \cite{MasuokaZubkov2011}.
The equality $\Ber(tI_{m+n})=t^{m-n}$ gives $C_{m|n}=\mu_{|m-n|}$ when $m\ne n$, and $C_{n|n}=D_{n|n}$.
Here $\mu_d$ is the group scheme of $d$th roots of unity, defined for $d\ge1$ by
\[
\mu_d:=\Ker\bigl(\Gm\longrightarrow\Gm;\quad t\longmapsto t^d\bigr).
\]

\begin{lemma}\label{lem:PGL_over_Z}
For $m,n\ge1$, $\PGLs(m|n)$ is obtained by base change from a split quasireductive supergroup over $\mathbb Z$.
The same holds for $\PSLs(m|n)$ when $m\ne n$.
\end{lemma}
\begin{proof}
Put
\[
H_{\mathbb Z}:=(\GLord_{m,\mathbb Z}\times\GLord_{n,\mathbb Z})/\mathrm G_{\mathrm m,\mathbb Z},\qquad
\mathfrak{pgl}(m|n)_{\mathbb Z}:=\mathfrak{gl}(m|n)_{\mathbb Z}/\mathbb Z I_{m+n},
\]
where the scalar torus is embedded diagonally.
With the induced conjugation action and matrix superbracket, $(H_{\mathbb Z},\mathfrak{pgl}(m|n)_{\mathbb Z})$ is a Harish--Chandra pair over $\mathbb Z$.
Then by \cite[Theorem~4.22]{MasuokaShibata2017}, we obtain a split quasireductive supergroup $\GZ$ over $\mathbb Z$.
One can easily see that the base change of $\GZ$ to $\Bbbk$ is isomorphic to $\PGLs(m|n)$.
For $m\ne n$, the composite
\[
\SLs(m|n)\longhookrightarrow\GLs(m|n)\longrightarrow\PGLs(m|n)
\]
induces an isomorphism $\PSLs(m|n)\cong\PGLs(m|n)$.
Thus $\GZ$ also gives a $\mathbb Z$-form of $\PSLs(m|n)$.
\end{proof}

\begin{proposition}\label{prop:central_forms}
For all $m,n\ge1$, $\PGLs(m|n)$ is an admissible type~$\mathrm I$ algebraic supergroup.
The same holds for $\PSLs(m|n)$ when $m\ne n$.
\end{proposition}
\begin{proof}
The assertion for $\PGLs(m|n)$ follows from \Cref{prop:central_quotient_general} applied to the scalar center $D_{m|n}$ of $\GLs(m|n)$.
For $m\ne n$, the isomorphism $\PSLs(m|n)\cong\PGLs(m|n)$ used in the proof of \Cref{lem:PGL_over_Z} gives the assertion for $\PSLs(m|n)$.
\end{proof}

\begin{samepage}
\begin{remark}\label{rem:equal_rank_typeA}
For every $n\ge1$, the supergroups $\SLs(n|n)$ and $\PSLs(n|n)$ do not satisfy the admissible type~$\mathrm I$ conditions.
On the diagonal torus of $\SLs(n|n)$, the Berezinian relation gives $\sum_{i,j}(\eps_i-\del_j)=0$.
Hence no homomorphism to $\mathbb R$ is positive on all weights of either odd matrix subspace.
For $n\ge3$, these subspaces are nonisomorphic simple $\SLord_n\times\SLord_n$-modules of dimension $n^2$.
A stable positive odd part would have to equal one of them, which is impossible by this relation.
For $n=2$, the relations
\[
\eps_1-\del_1=-(\eps_2-\del_2),\qquad
\eps_1-\del_2=-(\eps_2-\del_1)
\]
make the odd root spaces two-dimensional.
For $n=1$, the odd weights are zero.
The scalar quotient $\SLs(n|n)\to\PSLs(n|n)$ preserves the odd part and its torus weights, so the same conclusions hold for $\PSLs(n|n)$.

\end{remark}
\end{samepage}

\subsubsection{Type $\mathtt C$}\label{subsubsec:setup_typeC}
For $\ell\ge1$, let $V$ be the superspace with ordered homogeneous basis
\[
f_1,\ldots,f_\ell,\;f_{-1},\ldots,f_{-\ell},\;e_+,e_-,
\]
with the $f$'s even and the $e$'s odd.
Put $J:=\operatorname{diag}(J_1,J_0)$, where
\[
J_1:=\begin{pmatrix}0&I_\ell\\-I_\ell&0\end{pmatrix},\qquad
J_0:=\begin{pmatrix}0&1\\1&0\end{pmatrix}.
\]
For a commutative superalgebra $R$, we represent even automorphisms of $V\otimes R$ by matrices with coefficients on the right.
An element $g\in\GLs(2\ell|2)(R)$ and its supertranspose have the forms
\[
g=\begin{pmatrix}A&B\\C&D\end{pmatrix},\qquad
g^{\mathrm{st}}:=\begin{pmatrix}A^{\mathrm t}&C^{\mathrm t}\\-B^{\mathrm t}&D^{\mathrm t}\end{pmatrix},
\]
where the entries of $A,D$ are even, those of $B,C$ are odd, and $(-)^{\mathrm t}$ is the ordinary transpose.
Define
\[
\widetilde{\SpOs}(2\ell|2)(R):=\{g\in\GLs(2\ell|2)(R)\mid g^{\mathrm{st}}Jg=J\}.
\]
Since the Berezinian is invariant under supertranspose, $g^{\mathrm{st}}Jg=J$ gives $\Ber(g)^2=1$.
We define
\[
\SpOs(2\ell|2):=\Ker\bigl(\Ber:\widetilde{\SpOs}(2\ell|2)\to\mu_2\bigr).
\]
The even groups of $\widetilde{\SpOs}(2\ell|2)$ and $\SpOs(2\ell|2)$ are $\Spord_{2\ell}\times\Oord_2$ and $\Spord_{2\ell}\times\SOord_2$, respectively.
Here $\Oord_2$ is the orthogonal group for $J_0$, $\SOord_2=\{\operatorname{diag}(t,t^{-1})\mid t\in\Gm\}\cong\Gm$ is its special orthogonal subgroup, and $\Spord_{2\ell}$ is the symplectic group for $J_1$.
The even reflection interchanging $e_+$ and $e_-$ has Berezinian $-1$, so the map to $\mu_2$ is surjective.
Since $\Spord_{2\ell}\times\SOord_2$ is connected, $\SpOs(2\ell|2)$ is also connected.
We identify its Lie superalgebra with the matrix Lie superalgebra $\mathfrak{spo}(2\ell|2)$ given in \cite[Section~3.2]{ShuWang2008}.

\begin{lemma}\label{lem:OSP_over_Z}
The supergroup $\SpOs(2\ell|2)$ is obtained by base change from a split quasireductive supergroup over $\mathbb Z$ with even part $\Spord_{2\ell,\mathbb Z}\times\mathrm G_{\mathrm m,\mathbb Z}$.
\end{lemma}

\begin{proof}
Let $\mathfrak g_{\mathbb Z}$ be the $\mathbb Z$-span of the Chevalley basis given in \cite[Proposition~3.4]{ShuWang2008}.
With the conjugation action and the matrix superbracket, $(\Spord_{2\ell,\mathbb Z}\times\mathrm G_{\mathrm m,\mathbb Z},\mathfrak g_{\mathbb Z})$ is a Harish--Chandra pair over $\mathbb Z$.
By \cite[Theorem~4.22]{MasuokaShibata2017}, this pair defines a split quasireductive supergroup $\mathbf G_{\mathbb Z}$.
Its base change to $\Bbbk$ has the same Harish--Chandra pair as $\SpOs(2\ell|2)$, and hence is isomorphic to it.
\end{proof}

\begin{remark}\label{rem:Chevalley_typeC}
Fioresi and Gavarini \cite{FioresiGavarini2012} gave a super analogue of the Chevalley construction.
Given a finite-dimensional simple complex Lie superalgebra of one of the classical types treated there, a Chevalley basis, and a faithful finite-dimensional rational representation with an admissible lattice, their construction gives an algebraic supergroup over $\mathbb Z$, called a Chevalley supergroup.
For $\mathfrak{spo}(2\ell|2)$, take the Chevalley basis in \cite[Proposition~3.4]{ShuWang2008} and the natural representation on $\mathbb C^{2\ell|2}$ with its standard lattice $\mathbb Z^{2\ell|2}$.
The resulting Chevalley supergroup is isomorphic over $\mathbb Z$ to the supergroup constructed in the proof of \Cref{lem:OSP_over_Z}.
This follows by comparing their Harish--Chandra pairs, using the description in \cite[Section~6.2]{MasuokaShibata2017}.
\end{remark}

Take the diagonal maximal torus
\[
T=\{\operatorname{diag}(s_1,\ldots,s_\ell,s_1^{-1},\ldots,s_\ell^{-1},t,t^{-1})\}
\subseteq\SpOs(2\ell|2)_{\mathrm{ev}}.
\]
Write $\del_i$ and $\eps$ for the coordinate characters $s_i$ and $t$, respectively, so $\mathsf X(T)=\bigoplus_{i=1}^\ell\mathbb Z\del_i\oplus\mathbb Z\eps$.
Choose the positive roots
\[
\begin{aligned}
\Dzerop&=\{\del_i-\del_j,\ \del_i+\del_j\mid1\le i<j\le\ell\}\sqcup\{2\del_i\mid1\le i\le\ell\},\\
\Donep&=\{\eps-\del_i,\ \eps+\del_i\mid1\le i\le\ell\}.
\end{aligned}
\]
\begin{proposition}\label{prop:OSP_assumption}
With these positive roots, $\SpOs(2\ell|2)$ has opposite distinguished parabolics and an admissible ordering.
\end{proposition}

\begin{proof}
The odd root spaces are one-dimensional.
The spaces $\mathfrak g_{\pm1}$ are the natural $\Spord_{2\ell}$-module tensored with the $\SOord_2$-characters $\pm\eps$, so they are $\SpOs(2\ell|2)_{\mathrm{ev}}$-stable.
A direct calculation gives $[\mathfrak g_{\pm1},\mathfrak g_{\pm1}]=0$.
Thus \Cref{prop:opposite_parabolics_lie_criterion} gives the opposite distinguished parabolics.

Extend $\Upsilon$ from $\mathbb Z\Delta$ to $\mathsf X(T)$ and put $d_i:=\Upsilon(\del_i)$.
The chosen positive roots give $\Upsilon(\eps)>d_1>\cdots>d_\ell>0$.
Keeping the values $d_i$ fixed, decrease the value on $\eps$ to one below $-d_1$.
The even signs stay fixed, and the odd signs change in the order
\[
\eps-\del_1\prec\cdots\prec\eps-\del_\ell
\prec\eps+\del_\ell\prec\cdots\prec\eps+\del_1.
\]
Use the original $\Upsilon$ for $\Upsilon_{[0]}$.
Choosing a value on $\eps$ in each subsequent interval between the crossing values, and finally below $-d_1$, gives $\Upsilon_{[1]},\ldots,\Upsilon_{[2\ell]}$ after restriction to $\mathbb Z\Delta$.
For example, setting the value on $\eps$ to $0$ gives $\Upsilon_{[\ell]}$, for which
\[
\Upsilon_{[\ell]}(\eps-\del_i)=-d_i<0<d_i=\Upsilon_{[\ell]}(\eps+\del_i)
\quad(1\le i\le\ell).
\]
Thus the above ordering is admissible.
\end{proof}

\begin{remark}\label{rem:typeC_OSp_conventions}
Our notation $\SpOs(2\ell|2)$ follows Shu and Wang \cite{ShuWang2008} and is also used by Marko and Zubkov \cite{MarkoZubkovOSP}.
Serganova \cite{Serganova2014} writes $\mathbf{SOSP}(2,2\ell)$ for this connected supergroup and $\mathbf{OSP}(2,2\ell)$ for $\widetilde{\SpOs}(2\ell|2)$.
Her convention places the orthogonal factor first, corresponding to the even subgroups $\SOord_2\times\Spord_{2\ell}$ and $\Oord_2\times\Spord_{2\ell}$, respectively.

This distinction is relevant to the existence of opposite distinguished parabolics.
The reflection interchanging $e_+$ and $e_-$ exchanges the two nonisomorphic simple $\Spord_{2\ell}\times\SOord_2$-modules $\mathfrak g_1$ and $\mathfrak g_{-1}$.
Thus the odd part is irreducible under $\Spord_{2\ell}\times\Oord_2$, so $\widetilde{\SpOs}(2\ell|2)$ cannot have opposite distinguished parabolics.
\end{remark}

Let $D_\ell=\{\pm I_{2\ell+2}\}\cong\mu_2$ be the scalar subgroup of $\SpOs(2\ell|2)$, and put
\[
\mathbf{PSpO}(2\ell|2):=\SpOs(2\ell|2)/D_\ell,\qquad
\Tq:=T/D_\ell.
\]
\begin{lemma}\label{lem:OSP_quotient_over_Z}
The supergroup $\mathbf{PSpO}(2\ell|2)$ is obtained by base change from a split quasireductive supergroup over $\mathbb Z$.
\end{lemma}

\begin{proof}
Let $\mathfrak g_{\mathbb Z}$ be the Lie superalgebra considered in the proof of \Cref{lem:OSP_over_Z}.
Put $H_{\mathbb Z}:=(\Spord_{2\ell,\mathbb Z}\times\mathrm G_{\mathrm m,\mathbb Z})/D_{\ell,\mathbb Z}$, where $D_{\ell,\mathbb Z}\cong\mu_{2,\mathbb Z}$ is embedded by $z\mapsto(zI_{2\ell},z)$, and let $q_{\mathbb Z}$ be the quotient map.
The action on $W_{\mathbb Z}:=(\mathfrak g_{\mathbb Z})_{\bar1}$ descends to $H_{\mathbb Z}$.
Set $\mathfrak k_{\mathbb Z}:=\Lie(H_{\mathbb Z})\oplus W_{\mathbb Z}$.
The $H_{\mathbb Z}$-action on $W_{\mathbb Z}$ induces a $\Lie(H_{\mathbb Z})$-action, which we write as $X\cdot v$.
Use the usual bracket on $\Lie(H_{\mathbb Z})$, and define
\[
[X,v]_{\mathfrak k_{\mathbb Z}}:=X\cdot v,\quad
[v,w]_{\mathfrak k_{\mathbb Z}}:=\Lie(q_{\mathbb Z})([v,w]_{\mathfrak g_{\mathbb Z}})
\qquad(X\in\Lie(H_{\mathbb Z}),\ v,w\in W_{\mathbb Z}).
\]
With the adjoint action on $\Lie(H_{\mathbb Z})$ and the descended action on $W_{\mathbb Z}$, the pair $(H_{\mathbb Z},\mathfrak k_{\mathbb Z})$ is a Harish--Chandra pair over $\mathbb Z$.
By \cite[Theorem~4.22]{MasuokaShibata2017}, this pair defines a split quasireductive supergroup $\mathbf G_{\mathbb Z}$ over $\mathbb Z$.
Since $D_\ell$ is \'etale over $\Bbbk$, the base change of $\mathbf G_{\mathbb Z}$ to $\Bbbk$ has the same Harish--Chandra pair as $\mathbf{PSpO}(2\ell|2)$, and hence is isomorphic to it.
\end{proof}

\begin{proposition}\label{prop:typeC_adjoint_setup}
The supergroup $\mathbf{PSpO}(2\ell|2)$ is an admissible type~$\mathrm I$ algebraic supergroup.
Under the character map induced by $T\to\Tq$, its character lattice is
\[
\mathsf X(\Tq)=\left\{\sum_{i=1}^\ell b_i\del_i+a\eps\in\mathsf X(T)\ \middle|\ a+\sum_{i=1}^\ell b_i\equiv0\pmod2\right\}.
\]
\end{proposition}

\begin{proof}
The subgroup $D_\ell$ is purely even, central, and diagonalizable, so the first assertion follows from \Cref{prop:central_quotient_general,lem:OSP_quotient_over_Z}.
A character $\sum_i b_i\del_i+a\eps$ descends to $\Tq$ if and only if its value $(-1)^{a+\sum_i b_i}$ on $(-I_{2\ell},-1)$ is $1$.
\end{proof}

\section{Typicality and simplicity}\label{sec:typicality}
Let $\G$ be an admissible type~$\mathrm I$ algebraic supergroup with opposite distinguished parabolics $\Ppm=\Gev\ltimes\Vpm$.
Recall that $\Lambda_{\pm}=\bigwedge(\mathfrak g_{\pm1})=\hy(\Vpm)$.

\subsection{Induced modules and Kac functors}\label{subsec:induced_kac}
For $\lambda\in\mathsf X(T)$, let $\Bbbk_\lambda$ be the purely even one-dimensional $T$-supermodule of weight $\lambda$.
We regard $\Bbbk_\lambda$ as a $\Bminus$-supermodule through the quotient $\Bminus\to T$.
Put
\[
H^0_{\G}(\lambda):=\ind_{\Bminus}^{\G}(\Bbbk_\lambda),\qquad
H^0_{\Gev}(\lambda):=\ind_{\Bminus_{\mathrm{ev}}}^{\Gev}(\Bbbk_\lambda),
\]
and write the set of dominant weights for $\Gev$ (with respect to $\Dzerop$) as
\[
\mathsf X(T)^+:=\{\lambda\in\mathsf X(T)\mid
\langle\lambda,\alpha^\vee\rangle\ge0\text{ for all }\alpha\in\Dzerop\}.
\]
Here $\langle-,-\rangle$ is the character--cocharacter pairing, and $\alpha^\vee$ is the coroot of $\alpha$.

As in the Introduction, $\sRep(\mathbf H)$ denotes the category of finite-dimensional left $\mathbf H$-supermodules with parity-preserving morphisms.
For an ordinary algebraic group $H$, let $\Rep(H)$ denote the category of finite-dimensional left $H$-modules.
We regard $\Rep(H)$ as the full subcategory of $\sRep(H)$ consisting of supermodules concentrated in even parity.
Let $M\in\sRep(\Gev)$.
Regard $M$ as a $\Pplus$-supermodule through the quotient $\Pplus\to\Gev$ and as a $\Pminus$-supermodule through the quotient $\Pminus\to\Gev$.
Thus $\gpone$ acts trivially in the first case and $\gmone$ in the second.
Define the \emph{Kac functor} and the \emph{dual Kac functor} by
\[
\KacG(M):=\hy(\G)\otimes_{\hy(\Pplus)}M,\qquad
\coKacG(M):=\uHom_{\hy(\Pminus)}(\hy(\G),M).
\]
Here $\uHom_A(U,V)$ denotes the superspace of even and odd $A$-linear maps between left $A$-supermodules $U,V$.
For a homogeneous $f\in\coKacG(M)$, $\hy(\Pminus)$-linearity and the left $\hy(\G)$-action are given by
\[
f(ba)=(-1)^{|b||f|}bf(a),\qquad
(x\cdot f)(a)=(-1)^{|x|(|f|+|a|)}f(ax)
\]
for homogeneous $b\in\hy(\Pminus)$ and $a,x\in\hy(\G)$.

For a supermodule $M$, its parity shift $\Pi M$ has the same action and reversed parity.
\begin{lemma}\label{lem:Kac_properties}
The functors $\KacG,\coKacG:\sRep(\Gev)\to\sRep(\G)$ are exact and faithful and commute with the parity shift $\Pi$.
For $M\in\sRep(\Gev)$, they satisfy
$\dim\KacG(M)=\dim\coKacG(M)=2^N\dim M$.
Multiplication and restriction to $\Lambdaplus$ give $\Gev$-isomorphisms
\[
\KacG(M)|_{\Gev}\cong\Lambdaminus\otimes M,\qquad
\coKacG(M)|_{\Gev}\cong\Lambdaplus^*\otimes M.
\]
Both tensor products have diagonal $\Gev$-actions, with the adjoint action on $\Lambdaminus$ and the dual adjoint action on $\Lambdaplus^*$.
\end{lemma}

\begin{proof}
Both $\KacG(M)$ and $\coKacG(M)$ have $T$-actions induced by the adjoint action on $\hy(\G)$ and the action on $M$.
These actions are given by
\[
t\bigl(a\otimes_{\hy(\Pplus)}m\bigr)=\Ad(t)(a)\otimes_{\hy(\Pplus)}tm,\qquad
(t\cdot f)(a)=t f(\Ad(t^{-1})(a))
\]
for $t\in T$, $a\in\hy(\G)$, $m\in M$, and $f\in\coKacG(M)$.
These $T$-actions induce the respective $\hy(T)$-actions.
A finite-dimensional $\hy(\G)$-supermodule with a $T$-action inducing its $\hy(T)$-action is a $\G$-supermodule by \cite[Corollary~5.10]{MasuokaShibata2017}.
Both supermodules are finite-dimensional by \eqref{eq:pbw_modules}, so they define functors $\sRep(\Gev)\to\sRep(\G)$.

By \eqref{eq:pbw_modules}, multiplication and restriction give the stated tensor-product isomorphisms as $\hy(\Gev)$-supermodules.
They also respect the displayed $T$-actions and are therefore $\Gev$-equivariant.
Tensoring with either of these nonzero finite-dimensional superspaces is exact and faithful and gives the dimension formulas.
For $\KacG$, the parity identification is $\Pi(a\otimes_{\hy(\Pplus)}m)\mapsto a\otimes_{\hy(\Pplus)}\Pi m$.
For $\coKacG$, the parity identification sends $\Pi f$ to the map $a\mapsto(-1)^{|a|}\Pi f(a)$.
\end{proof}

By \cite[Lemma~5.3]{Shibata2020} and transitivity of induction, we have $H^0_{\G}(\lambda)\cong\ind_{\Pminus}^{\G}H^0_{\Gev}(\lambda)$ as $\G$-supermodules.
The following lemma rewrites this isomorphism using the dual Kac functor.

\begin{lemma}\label{thm:BWtensor}
For every $\lambda\in\mathsf X(T)$, there is a $\G$-isomorphism
\[
H^0_{\G}(\lambda)\cong\coKacG(H^0_{\Gev}(\lambda)).
\]
\end{lemma}

\begin{proof}
Transitivity of induction gives $H^0_{\G}(\lambda)\cong\ind_{\Pminus}^{\G}\bigl(\ind_{\Bminus}^{\Pminus}\Bbbk_\lambda\bigr)$.
Since $\Voddminus$ acts trivially on $\Bbbk_\lambda$, $\Pminus/\Voddminus\cong\Gev$, and $\Bminus/\Voddminus\cong\Bminus_{\mathrm{ev}}$, restriction of functions gives a $\Pminus$-isomorphism $\ind_{\Bminus}^{\Pminus}\Bbbk_\lambda\cong H^0_{\Gev}(\lambda)=:M$, where $M$ is regarded as a $\Pminus$-supermodule through the quotient $\Pminus\to\Gev$.
For every finite-dimensional $\G$-supermodule $V$, Frobenius reciprocity and hyperalgebra coinduction give isomorphisms natural in $V$
\[
\Hom_{\G}(V,\ind_{\Pminus}^{\G}M)
\cong\Hom_{\Pminus}(V|_{\Pminus},M)
\cong\Hom_{\G}(V,\coKacG(M)).
\]
The last isomorphism is given by $h\mapsto\widehat h$, where $\widehat h(v)(a)=(-1)^{|a||v|}h(av)$ for homogeneous $a\in\hy(\G)$ and $v\in V$.
The map $\widehat h$ is $\hy(\G)$-linear and respects the $T$-action in the proof of \Cref{lem:Kac_properties}, so it is $\G$-equivariant.
Thus $\ind_{\Pminus}^{\G}M\cong\coKacG(M)$, which proves the lemma.
\end{proof}

Ordinary modules over a purely even group are regarded as supermodules concentrated in even parity.
For $\lambda\in\mathsf X(T)^+$, put
$L_{\G}(\lambda):=\soc_{\G}H^0_{\G}(\lambda)$ and
$L_{\Gev}(\lambda):=\soc_{\Gev}H^0_{\Gev}(\lambda)$.
By \cite{Shibata2020}, $L_{\G}(\lambda)$ is simple, and the modules $L_{\G}(\lambda)$ and $\Pi L_{\G}(\lambda)$, for $\lambda\in\mathsf X(T)^+$, give all simple $\G$-supermodules.
The module $L_{\Gev}(\lambda)$ is the simple $\Gev$-module of highest weight $\lambda$.
\subsection{Typical weights}\label{subsec:typical_weights}
Fix an admissible ordering $\beta_1\prec\cdots\prec\beta_N$ of $\Donep$.
Choose nonzero root vectors $X_\beta\in\mathfrak g_{\bar1}^{\beta}$ and
$Y_\beta\in\mathfrak g_{\bar1}^{-\beta}$, and put
$H_\beta:=[X_\beta,Y_\beta]\in\mathfrak h$.

Put
\[
\rho_{\bar0}:=\frac12\sum_{\alpha\in\Dzerop}\alpha,\qquad
\rho_{[i]}:=\rho_{\bar0}-\frac12\sum_{\gamma\in\Donepi{i}}\gamma
\quad(0\le i\le N),
\]
and, for simplicity, write $\rho:=\rho_{[0]}=\frac12\sum_{\alpha\in\Dzerop}\alpha-\frac12\sum_{\beta\in\Donep}\beta$.

\begin{definition}\label{def:typical}
For a weight $\lambda\in\mathsf X(T)$, we define its \emph{typicality scalar} by
\[
\Dtyp(\lambda):=\prod_{\beta\in\Donep}(\lambda+\rho)(H_\beta)\in\Bbbk.
\]
We call $\lambda$ \emph{typical} if $\Dtyp(\lambda)\ne0$ and \emph{atypical} otherwise.
\end{definition}

The scalar $\Dtyp(\lambda)$ does not depend on the admissible ordering.
Rescaling the root vectors multiplies it by a nonzero scalar and therefore leaves typicality unchanged.

\begin{lemma}\label{lem:rho_compat}
For $1\le i\le N$, we have $\rho_{[i-1]}(H_{\beta_i})=0$ in $\Bbbk$.
\end{lemma}

\begin{proof}
Fix $i$ and consider
\[
\mathfrak k:=\mathfrak h\oplus
\bigoplus_{\gamma\in\Delta^+(\Upsilon_{[i-1]})}\mathfrak g^\gamma
\oplus\mathfrak g^{-\beta_i}.
\]
This superspace is stable under $\ad(X_{\beta_i})$ and $\ad(Y_{\beta_i})$.
Indeed, both operators send $\mathfrak h\oplus\mathfrak g^{-\beta_i}$ into $\mathfrak k$, since $-2\beta_i$ is not an even root by \Cref{lem:h1_and_abelian}.
If $\gamma\in\Delta^+(\Upsilon_{[i-1]})$ and $\gamma+\beta_i$ is a root, then it is positive for $\Upsilon_{[i-1]}$.
If $\gamma\ne\beta_i$ and $\gamma-\beta_i$ is a root, then it is positive for $\Upsilon_{[i]}$ and is different from $-\beta_i$.
It is therefore also positive for $\Upsilon_{[i-1]}$.
For $\gamma=\beta_i$, the bracket with $Y_{\beta_i}$ belongs to $\mathfrak h$.

The root spaces occurring in $\mathfrak k$ are one-dimensional.
Taking the supertrace of the supercommutator gives
\[
0=\str_{\mathfrak k}[\ad(X_{\beta_i}),\ad(Y_{\beta_i})]
=(2\rho_{[i-1]})(H_{\beta_i})+\beta_i(H_{\beta_i}).
\]
By \Cref{lem:h1_and_abelian} and the super Jacobi identity,
\[
2[H_{\beta_i},X_{\beta_i}]
=[Y_{\beta_i},[X_{\beta_i},X_{\beta_i}]]=0.
\]
Hence $\beta_i(H_{\beta_i})=0$, and the assertion follows.
\end{proof}

Let $v_\lambda$ be a highest weight vector of weight $\lambda$ in a $\G$-supermodule $V$.
Put $\lambda_{[0]}:=\lambda$, $v_{[0]}:=v_\lambda$, and
\[
\lambda_{[i]}:=\lambda-\sum_{j=1}^{i}\beta_j,\qquad
v_{[i]}:=Y_{\beta_i}\cdots Y_{\beta_1}v_{[0]}
\quad(1\le i\le N).
\]

\begin{lemma}\label{lem:oddflipvectors}
For $0\le i\le N$, we have $v_{[i]}\in V_{\lambda_{[i]}}$, and $v_{[i]}$ is annihilated by every $\mathfrak g^\gamma$ with $\gamma\in\Delta^+(\Upsilon_{[i]})$.
For $1\le i\le N$, it also satisfies
\[
X_{\beta_i}v_{[i]}=\lambda_{[i-1]}(H_{\beta_i})v_{[i-1]}.
\]
If $\lambda\in\mathsf X(T)^+$, $V=\KacG(L_{\Gev}(\lambda))$, and $v_\lambda=1\otimes_{\hy(\Pplus)}w_\lambda$ for a highest weight vector $w_\lambda\in L_{\Gev}(\lambda)$, then every $v_{[i]}$ is nonzero.
\end{lemma}

\begin{proof}
The weight assertion follows from the definition of $v_{[i]}$.
For the annihilation assertion, argue by induction on $i$.
For $i=0$, the assertion holds because $v_\lambda$ is a highest weight vector.
For the new positive root $-\beta_i$, we have
$Y_{\beta_i}v_{[i]}=Y_{\beta_i}^2v_{[i-1]}=0$ by \Cref{lem:h1_and_abelian}.
If $\gamma\in\Delta^+(\Upsilon_{[i]})\setminus\{-\beta_i\}$ and $Z\in\mathfrak g^\gamma$, then
\[
ZY_{\beta_i}v_{[i-1]}
=[Z,Y_{\beta_i}]v_{[i-1]}+(-1)^{|Z|}Y_{\beta_i}Zv_{[i-1]}.
\]
The root $\gamma$ is also positive for $\Upsilon_{[i-1]}$, so the second term is zero by induction.
If $\gamma-\beta_i$ is a root, then it is positive for $\Upsilon_{[i]}$ because
$\Upsilon_{[i]}(\gamma)>0>\Upsilon_{[i]}(\beta_i)$.
It differs from $-\beta_i$ and is therefore positive for $\Upsilon_{[i-1]}$ as well.
Thus the first term vanishes by induction.
If $\gamma-\beta_i$ is not a root, the bracket is zero.
Since $\beta_i\in\Delta^+(\Upsilon_{[i-1]})$, the induction hypothesis gives $X_{\beta_i}v_{[i-1]}=0$.
Hence
\[
X_{\beta_i}v_{[i]}
=[X_{\beta_i},Y_{\beta_i}]v_{[i-1]}
=\lambda_{[i-1]}(H_{\beta_i})v_{[i-1]}.
\]
In the stated Kac module, the PBW expression $v_{[i]}=Y_{\beta_i}\cdots Y_{\beta_1}\otimes_{\hy(\Pplus)}w_\lambda$ proves nonvanishing.
\end{proof}

Recall that $\beta_1\prec\cdots\prec\beta_N$ is the fixed admissible ordering of $\Donep$.
As in \eqref{eq:intro_EF}, put
\[
E:=X_{\beta_1}\cdots X_{\beta_N},\qquad
F:=Y_{\beta_N}\cdots Y_{\beta_1}
\]
in $\hy(\G)$.
They span the top exterior powers
$\Lambdaplus^N:=\bigwedge^N(\gpone)$ and
$\Lambdaminus^N:=\bigwedge^N(\gmone)$, respectively.

\begin{theorem}\label{thm:odd_product}
Let $v_\lambda$ be a highest weight vector of weight $\lambda$ in a $\G$-supermodule.
Then
\[
EFv_\lambda=\Dtyp(\lambda)v_\lambda,\qquad
\Dtyp(\lambda)=\prod_{i=1}^{N}\lambda_{[i-1]}(H_{\beta_i}),
\]
where $\lambda_{[i]}:=\lambda-\sum_{j=1}^i\beta_j$.
\end{theorem}

\begin{proof}
Applying \Cref{lem:oddflipvectors} successively gives
$EFv_\lambda=\bigl(\prod_i\lambda_{[i-1]}(H_{\beta_i})\bigr)v_\lambda$.
The identities $\rho_{[i]}=\rho_{[i-1]}+\beta_i$ and
$\lambda_{[i]}=\lambda_{[i-1]}-\beta_i$ show that
$\lambda_{[i-1]}+\rho_{[i-1]}=\lambda+\rho$.
Since $\rho_{[i-1]}(H_{\beta_i})=0$ by \Cref{lem:rho_compat}, each factor is
$(\lambda+\rho)(H_{\beta_i})$.
\end{proof}

\subsection{Comparison of the Kac functors}\label{subsec:kac_simple}\label{subsec:canonical_map}

Using the PBW decomposition \eqref{eq:pbw_decomposition}, define
\[
\pr:\hy(\G)\cong\Lambdaminus\otimes\hy(\Gev)\otimes\Lambdaplus
\xrightarrow{\ \varepsilon'\otimes\id\otimes\varepsilon\ }\hy(\Gev).
\]
Here $\varepsilon'$ and $\varepsilon$ are the counits of $\Lambdaminus$ and $\Lambdaplus$, respectively.
Put
\[
\Theta:=\pr(EF),
\]
and let $\Theta_M:M\to M$ denote its action on a left $\hy(\Gev)$-supermodule $M$.
The definitions of $\KacG(M)$ and $\coKacG(M)$ in \Cref{subsec:induced_kac} also apply to an arbitrary left $\hy(\Gev)$-supermodule $M$, with $\gpone$ and $\gmone$ acting trivially, respectively.
The resulting supermodules are left $\hy(\G)$-supermodules.

The element $\Theta$ is the analogue of the central element constructed by Chen--Mazorchuk \cite[Lemma~6.5]{ChenMazorchuk2021}.
The proof of centrality below uses the same PBW argument, formulated here in terms of the adjoint action of $\Gev$.

\begin{lemma}\label{lem:centralTheta}\label{cor:EF_on_all_M}
The element $EF$ is invariant under the adjoint action of $\Gev$, and $\Theta$ is central in $\hy(\Gev)$.
For a $\Gev$-supermodule $M$, the map $\Theta_M$ is a $\Gev$-endomorphism, natural in $M$.
For every left $\hy(\Gev)$-supermodule $M$, we have
\[
EF(1\otimes_{\hy(\Pplus)}m)=1\otimes_{\hy(\Pplus)}\Theta_Mm\qquad(m\in M)
\]
in $\KacG(M)$.
For every $\lambda\in\mathsf X(T)^+$, the element $\Theta$ acts on $L_{\Gev}(\lambda)$ by $\Dtyp(\lambda)$.
\end{lemma}

\begin{proof}
The one-dimensional $\Gev$-supermodules
$\Bbbk E=\Lambdaplus^N$ and $\Bbbk F=\Lambdaminus^N$ have opposite $T$-weights.
Since restriction $\mathsf X(\Gev)\to\mathsf X(T)$ is injective, their characters are inverse.
Multiplication is equivariant, so $EF$ is $\Gev$-invariant and commutes with $\hy(\Gev)$.
The projection $\pr$ is $\Gev$-equivariant because all three factors in the PBW decomposition are $\Gev$-stable.
Consequently $\Theta$ is invariant and central, and its action gives the stated natural endomorphism.

Give $\gpone$, $\gmone$, and $\hy(\Gev)$ the integer degrees $1$, $-1$, and $0$, respectively.
The PBW relations preserve these degrees because $\gpone$ and $\gmone$ are $\Gev$-stable and
$[\gpone,\gmone]\subseteq\mathfrak g_{\bar0}$.
For subsets $S=\{s_1<\cdots<s_d\}$ and $J=\{j_1<\cdots<j_e\}$ of $\{1,\ldots,N\}$, put
$Y_S:=Y_{\beta_{s_d}}\cdots Y_{\beta_{s_1}}$ and
$X_J:=X_{\beta_{j_1}}\cdots X_{\beta_{j_e}}$.
Since $EF$ has degree zero, its PBW expansion is
\begin{equation}\label{eq:Theta0_PBW}
EF=\Theta+\sum_{|S|=|J|>0}Y_Su_{S,J}X_J,
\end{equation}
where $u_{S,J}\in\hy(\Gev)$.
Each term in the sum has $X_J\ne1$ on the right, so it annihilates $1\otimes_{\hy(\Pplus)}M$.
This proves the formula for $EF$.
For $M=L_{\Gev}(\lambda)$, \Cref{thm:odd_product} gives the scalar on the highest weight line $1\otimes_{\hy(\Pplus)}M_\lambda$.
Since $M_\lambda$ generates $M$ as a $\hy(\Gev)$-module, centrality gives the same scalar on all of $M$.
\end{proof}

For a $\gpone$-supermodule $V$, put $V^{\gpone}:=\{v\in V\mid Xv=0\text{ for all }X\in\gpone\}$.

\begin{proposition}\label{lem:canonical_map}\label{cor:Kac_dualKac_block}
Let $M$ be a finite-dimensional $\Gev$-supermodule.
\begin{enumerate}[label=(\roman*)]
\item There is an even $\G$-homomorphism
$\can_M:\KacG(M)\to\coKacG(M)$, natural in $M$ and nonzero for $M\ne0$.
\item The map $\can_M$ is an isomorphism if and only if $\Theta_M$ is invertible.
\item If $\Theta_M$ is invertible, then
\[
M\longrightarrow\KacG(M)^{\gpone};\quad m\longmapsto1\otimes_{\hy(\Pplus)}m
\]
is an isomorphism of $\Gev$-supermodules.
\end{enumerate}
The same assertions hold for finite-dimensional left $\hy(\Gev)$-supermodules.
In this case, $\can_M$ is $\hy(\G)$-linear and the map in \textup{(iii)} is $\hy(\Gev)$-linear.
\end{proposition}

\begin{proof}
Under the PBW decomposition $\hy(\G)\cong\hy(\Pminus)\otimes\Lambdaplus$, define
\[
\pr^{<0}:\hy(\G)\longrightarrow\hy(\Pminus);\quad
b\otimes\xi\longmapsto b\varepsilon(\xi),
\]
where $\varepsilon$ is the counit of $\Lambdaplus$.
This projection is left $\hy(\Pminus)$-linear, right $\hy(\Gev)$-linear, and $T$-equivariant, with kernel $\hy(\G)\gpone$.
For homogeneous $m\in M$, define $j_M(m):\hy(\G)\to M$ by
\[
j_M(m)(a):=(-1)^{|a||m|}\pr^{<0}(a)m
\]
for homogeneous $a\in\hy(\G)$.
Left $\hy(\Pminus)$-linearity of $\pr^{<0}$ gives $j_M(m)(ba)=(-1)^{|b||m|}b\,j_M(m)(a)$ for homogeneous $b\in\hy(\Pminus)$.
Thus $j_M:M\to\coKacG(M)$ is an even map.
Right $\hy(\Gev)$-linearity of $\pr^{<0}$ and the equality $\pr^{<0}(aX)=0$ for $X\in\gpone$ imply
\[
u\cdot j_M(m)=j_M(um),\qquad X\cdot j_M(m)=0
\quad(u\in\hy(\Gev),\ X\in\gpone).
\]
Hence $\can_M(a\otimes_{\hy(\Pplus)}m):=a\cdot j_M(m)$ defines an even $\hy(\G)$-homomorphism.
When $M$ is a $\Gev$-supermodule, $\can_M$ also commutes with the $T$-actions, since $\pr^{<0}$ does.
Thus it is a $\G$-homomorphism.
The construction is natural in $M$, and $j_M(m)(1)=m$ proves (i).

For (ii), suppose first that $\Theta_M$ is invertible and that the kernel of $\can_M$ contains a nonzero homogeneous element
$z=\sum_S Y_S\otimes_{\hy(\Pplus)}m_S$.
Choose $S$ with $m_S\ne0$ and $|S|$ minimal.
Multiplication by $Y_{S^c}$, where $S^c:=\{1,\ldots,N\}\setminus S$, gives $\pm F\otimes_{\hy(\Pplus)}m_S\ne0$, still in the kernel.
Indeed, $Y_{S^c}Y_J=0$ unless $J\subseteq S$.
If $J\subseteq S$ and $m_J\ne0$, the choice of $S$ gives $|J|=|S|$, hence $J=S$.
On the other hand, the coinduced action and \eqref{eq:Theta0_PBW} give
\begin{equation}\label{eq:canonical_top_evaluation}
\can_M(F\otimes_{\hy(\Pplus)}m)(E)=(-1)^{N|m|+N}\Theta_Mm
\end{equation}
for homogeneous $m\in M$.
For $m=m_S$, the displayed value is nonzero because $\Theta_M$ is invertible, contradicting $F\otimes_{\hy(\Pplus)}m_S\in\Ker(\can_M)$.
Thus $\can_M$ is injective and, by equality of dimensions, an isomorphism.

Conversely, let $0\ne m\in\Ker(\Theta_M)$ be homogeneous.
The map $\can_M(F\otimes_{\hy(\Pplus)}m)$ vanishes on every monomial $X_J$ in $\Lambdaplus$.
For $|J|<N$, every term of $\pr^{<0}(X_JF)$ has integer degree $|J|-N<0$ and therefore lies in the ideal of $\hy(\Pminus)$ generated by $\gmone$, which acts trivially on $M$.
For $|J|=N$, the value is zero by \eqref{eq:canonical_top_evaluation}.
By $\hy(\Pminus)$-linearity, this map also vanishes on every product $bX_J$ with $b\in\hy(\Pminus)$.
These products span $\hy(\G)$ by the PBW decomposition, so $\can_M(F\otimes_{\hy(\Pplus)}m)=0$.
Thus $\can_M$ is not injective.
Since $M$ is finite-dimensional, $\Theta_M$ is noninvertible exactly when its kernel is nonzero.
This proves (ii).

An element of $\coKacG(M)$ is annihilated by $\gpone$ exactly when its restriction to $\Lambdaplus$ vanishes in positive exterior degree.
Evaluation at $1$ therefore identifies $\coKacG(M)^{\gpone}$ with $M$, with inverse $j_M$.
Since $\can_M(1\otimes_{\hy(\Pplus)}m)=j_M(m)$, this proves (iii).
For a left $\hy(\Gev)$-supermodule, the same proof applies without the statements about $T$-equivariance.
\end{proof}

\begin{theorem}\label{thm:Kac_simple}\label{thm:dualKac_simple}
For a finite-dimensional $\Gev$-supermodule $M$, the following conditions are equivalent.
\begin{enumerate}[label=(\roman*)]
\item $\KacG(M)$ is simple.
\item $\coKacG(M)$ is simple.
\item $M$ is simple and $\Theta_M$ is invertible.
\end{enumerate}
\end{theorem}

\begin{proof}
Exactness and faithfulness of $\KacG$ and $\coKacG$ in \Cref{lem:Kac_properties} show that either (i) or (ii) forces $M$ to be simple.
The nonzero map $\can_M$ is then injective in case (i), or surjective in case (ii).
Since its source and target have equal dimension, it is an isomorphism.
By \Cref{lem:canonical_map}, $\Theta_M$ is invertible.

Conversely, suppose that (iii) holds, and let
$0\ne V\subseteq\KacG(M)$ be a $\G$-subsupermodule.
Multiplying a nonzero homogeneous vector in $V$ by suitable negative odd root vectors, as in the proof of \Cref{lem:canonical_map}, gives
$V\cap(\Bbbk F\otimes M)\ne0$.
Since $M$ is simple and $\Bbbk F$ is one-dimensional, $\Bbbk F\otimes M$ is a simple $\Gev$-supermodule.
The intersection $V\cap(\Bbbk F\otimes M)$ is a nonzero $\Gev$-subsupermodule, so $V\cap(\Bbbk F\otimes M)=\Bbbk F\otimes M$.
Applying $E$ and using \Cref{cor:EF_on_all_M} gives
$1\otimes_{\hy(\Pplus)}M\subseteq V$, because $\Theta_M$ is invertible.
These vectors generate $\KacG(M)$, proving (i).
The isomorphism $\can_M$ then proves (ii).
\end{proof}

\Cref{thm:Kac_simple} gives a simplicity criterion for Kac modules induced from finite-dimensional modules for $\Gev$.
In characteristic zero, Chen and Mazorchuk proved a corresponding criterion for certain classical type~$\mathrm I$ Lie superalgebras, allowing the simple module for the even part to be infinite-dimensional \cite{ChenMazorchuk2021}.

For $\lambda\in\mathsf X(T)^+$, \Cref{cor:EF_on_all_M} gives $\Theta_{L_{\Gev}(\lambda)}=\Dtyp(\lambda)\id$.
Thus $\Theta_{L_{\Gev}(\lambda)}$ is invertible if and only if $\lambda$ is typical.
By \Cref{thm:Kac_simple}, we have
\[
\KacG(L_{\Gev}(\lambda))\text{ is simple}\quad\Longleftrightarrow\quad\coKacG(L_{\Gev}(\lambda))\text{ is simple}\quad\Longleftrightarrow\quad\lambda\text{ is typical}.
\]
More precisely, the following holds.

\begin{corollary}\label{cor:typical_simple_kac_realization}
For a typical $\lambda\in\mathsf X(T)^+$, there are $\G$-isomorphisms
\[
L_{\G}(\lambda)\cong\KacG(L_{\Gev}(\lambda))
\cong\coKacG(L_{\Gev}(\lambda)).
\]
\end{corollary}

\begin{proof}
Exactness of $\coKacG$ and \Cref{thm:BWtensor} give an embedding
\[
\coKacG(L_{\Gev}(\lambda))\hookrightarrow
\coKacG(H^0_{\Gev}(\lambda))\cong H^0_{\G}(\lambda).
\]
Its source is simple by \Cref{thm:dualKac_simple}, so its image is the simple socle $L_{\G}(\lambda)$.
The other isomorphism is $\can_{L_{\Gev}(\lambda)}$.
\end{proof}

\begin{remark}\label{rem:LiShu_typical_realization}
Li and Shu also considered this realization for $\GLs(m|n)$ in positive characteristic $p$ \cite[Theorem~10.3]{LiShu2023}.
Their $p$-typical condition is our typicality condition.
Note that with their definition of $\mathcal K(\mu)$ in \cite[Section~10.1]{LiShu2023}, the isomorphism in \cite[Theorem~10.3(iii)]{LiShu2023} should read $L(\lambda)\cong\mathcal K(\lambda)$, without the weight shift.
This formula is the $\GLs(m|n)$ case of the first isomorphism in \Cref{cor:typical_simple_kac_realization}.
\end{remark}

\subsection{The simplicity theorem}\label{subsec:irreducibility}
The following is our first main theorem, extending Marko's simplicity criterion for $\GLs(m|n)$ \cite{Marko2018} to every admissible type~$\mathrm I$ supergroup.

\begin{theorem}\label{thm:main_irred}
For every $\lambda\in\mathsf X(T)^+$,
\[
H^0_{\G}(\lambda)\text{ is simple}
\quad\Longleftrightarrow\quad
H^0_{\Gev}(\lambda)\text{ is simple and }\lambda\text{ is typical}.
\]
\end{theorem}

\begin{proof}
By \Cref{thm:BWtensor}, $H^0_{\G}(\lambda)\cong\coKacG(H^0_{\Gev}(\lambda))$.
Apply \Cref{thm:dualKac_simple}.
If $H^0_{\Gev}(\lambda)$ is simple, it equals its socle $L_{\Gev}(\lambda)$.
Now \Cref{cor:EF_on_all_M} identifies its $\Theta$-action with multiplication by $\Dtyp(\lambda)$.
\end{proof}

In \Cref{subsec:typeA}, we identify $\Dtyp(\lambda)$ with Marko's product of linear factors for $\GLs(m|n)$.

For a finite-dimensional $T$-supermodule $M$, let $\ch(M)$ be its ordinary formal character.
Thus $\ch(\Pi M)=\ch(M)$.
For a typical $\lambda\in\mathsf X(T)^+$, \Cref{cor:typical_simple_kac_realization,lem:Kac_properties} give
\begin{equation}\label{eq:typical_simple_character}
\ch\bigl(L_{\G}(\lambda)\bigr)
=\ch\bigl(L_{\Gev}(\lambda)\bigr)
\prod_{\beta\in\Donep}(1+e^{-\beta}),\quad
\dim L_{\G}(\lambda)=2^N\dim L_{\Gev}(\lambda).
\end{equation}

\begin{corollary}\label{cor:charzero}
Assume that $\operatorname{char}\Bbbk=0$.
For every $\lambda\in\mathsf X(T)^+$, the module $H^0_{\G}(\lambda)$ is simple if and only if $\lambda$ is typical.
\end{corollary}

\begin{proof}
The induced module $H^0_{\Gev}(\lambda)$ is simple for the split connected reductive group $\Gev$ in characteristic zero.
Apply \Cref{thm:main_irred}.
\end{proof}

The following construction gives admissible type~$\mathrm I$ supergroups that need not be of types~$\mathtt A$ or~$\mathtt C$.
\begin{example}\label{ex:central_odd_bracket}
Let $H_{\mathbb Z}$ be a split connected reductive group scheme with a split maximal torus $T_{\mathbb Z}$.
Let $W_{\mathbb Z}$ be a finite free $H_{\mathbb Z}$-module of rank $N>0$, and assume that each weight space occurring in $W_{\mathbb Z}$ has rank one.
Write $H$, $T$, and $W$ for their base changes to $\Bbbk$.
Set $G:=\Gm\times\Gm\times H$, and let $\tau$ be the standard character of the second factor $\Gm$ of $G$.
Take $\mathfrak g_{\bar0}:=\Lie(G)$, $\gpone:=\Bbbk_\tau\otimes W$, and $\gmone:=\Bbbk_{-\tau}\otimes W^*$, with the first factor $\Gm$ of $G$ acting trivially on the odd part.
Let $c$ be the standard generator of the Lie algebra of the first factor $\Gm$ of $G$.
Identify the underlying odd spaces with $W$ and $W^*$, and define the odd bracket by
\[
[v,f]=f(v)c,\qquad
[\gpone,\gpone]=[\gmone,\gmone]=0
\quad(v\in W,\ f\in W^*).
\]
The evaluation pairing is $G$-invariant, and $c$ acts trivially on the odd part.
Thus the above brackets define a Lie superalgebra structure on $\mathfrak g_{\bar0}\oplus\gpone\oplus\gmone$.
Over $\mathbb Z$, take the even group $\mathrm G_{\mathrm m,\mathbb Z}^{2}\times H_{\mathbb Z}$, and let $c_{\mathbb Z}$ be the standard generator of the Lie algebra of its first $\mathrm G_{\mathrm m,\mathbb Z}$ factor.
Under base change to $\Bbbk$, the element $c_{\mathbb Z}$ maps to $c$.
On the odd part $W_{\mathbb Z}\oplus W_{\mathbb Z}^{*}$, with the actions described above, define $[v,f]=f(v)c_{\mathbb Z}$ and $[W_{\mathbb Z},W_{\mathbb Z}]=[W_{\mathbb Z}^{*},W_{\mathbb Z}^{*}]=0$ for $v\in W_{\mathbb Z}$ and $f\in W_{\mathbb Z}^{*}$.
Since $[v+f,v+f]=2f(v)c_{\mathbb Z}$, the required $2$-divisibility condition holds.
These data define a Harish--Chandra pair over $\mathbb Z$.
This pair gives a split quasireductive supergroup $\GZ$ with split maximal torus $\mathrm G_{\mathrm m,\mathbb Z}^{2}\times T_{\mathbb Z}$.
Its base change $\G$ has $\Gev=G$.

Let $\Omega$ be the set of $T$-weights of $W$.
We use $\Gm^2\times T$ as a split maximal torus of $G$.
The even roots of $\G$ are the roots of $(H,T)$, and the odd roots are $\pm(\tau+\nu)$ for $\nu\in\Omega$.
Fix a positive system $\Delta^+(H,T)$.
Choose a real linear functional $\Upsilon$ on $\mathsf X(T)\otimes_{\mathbb Z}\mathbb R$ that is positive on this system and takes distinct values on $\Omega$.
Extend it to $\mathsf X(\Gm^2\times T)\otimes_{\mathbb Z}\mathbb R$ so that $\Upsilon(\tau+\nu)>0$ for all $\nu\in\Omega$.
The resulting positive roots are $\Dzerop=\Delta^+(H,T)$ and $\Donep=\{\tau+\nu\mid\nu\in\Omega\}$.
All odd root spaces are one-dimensional, and $\mathfrak g_{\pm1}$ are $G$-stable and abelian.
Hence \Cref{prop:opposite_parabolics_lie_criterion} gives opposite distinguished parabolics of $\G$.
Enumerate $\Omega=\{\nu_1,\ldots,\nu_N\}$ so that $\Upsilon(\nu_1)<\cdots<\Upsilon(\nu_N)$, and put $\beta_i:=\tau+\nu_i$.
For each $0\le i\le N$, the set
\[
\Delta^+(H,T)\cup\{-\beta_1,\ldots,-\beta_i,\beta_{i+1},\ldots,\beta_N\}
\]
is a positive root set of $\G$.
Hence $\beta_1\prec\cdots\prec\beta_N$ is an admissible ordering, and $\G$ is admissible of type~$\mathrm I$.

Choose root vectors from a weight basis of $W_{\mathbb Z}$ and its dual basis.
Then $H_\beta=c$ for every $\beta\in\Donep$.
Since all roots vanish on $c$, we obtain
\[
\Dtyp(\lambda)
=\prod_{\beta\in\Donep}(\lambda+\rho)(H_\beta)
=\lambda(c)^N
\qquad\bigl(\lambda\in\mathsf X(\Gm^2\times T)^+\bigr).
\]
Consequently, $\lambda$ is typical if and only if $\lambda(c)\ne0$ in $\Bbbk$.
For example, $H_{\mathbb Z}=\SLord_{2,\mathbb Z}$ and $W_{\mathbb Z}=\operatorname{Sym}^3(\mathbb Z^2)$ give $N=4$ and $\Dtyp(\lambda)=\lambda(c)^4$.
\end{example}

\section{Typical blocks and Frobenius kernels}\label{sec:general-applications}
Let $\G$ be an admissible type~$\mathrm I$ algebraic supergroup with opposite distinguished parabolics $\Ppm=\Gev\ltimes\Vpm$.

\subsection{Typical blocks}\label{subsec:serre_equiv}

For $U\in\sRep(\Gev)$, let $\Theta_U$ be the endomorphism defined in \S\ref{subsec:canonical_map}, and put $d:=\dim U$.
The Fitting decomposition gives
\[
U=\Ker(\Theta_U^d)\oplus U_{\mathrm{inv}},\qquad
U_{\mathrm{inv}}:=\operatorname{Im}(\Theta_U^d).
\]
The restriction of $\Theta_U$ is nilpotent on $\Ker(\Theta_U^d)$ and invertible on $U_{\mathrm{inv}}$.
Both summands are $\Gev$-stable by \Cref{lem:centralTheta}.
Every $\Gev$-homomorphism commutes with $\Theta$, so these decompositions and their projections are natural in $U$.
Since the maps in every short exact sequence commute with these projections, the two summand functors are exact.

Let $\mathscr E$ be the full subcategory of $\sRep(\Gev)$ on which $\Theta$ acts invertibly.
Let $\mathscr C^{\mathrm{typ}}$ be the Serre subcategory of $\sRep(\G)$ generated by $L_{\G}(\lambda)$ and $\Pi L_{\G}(\lambda)$ for typical $\lambda\in\mathsf X(T)^+$.
By \Cref{cor:EF_on_all_M}, the objects of $\mathscr E$ are precisely the $\Gev$-supermodules whose simple composition factors have typical highest weights.

For $W\in\sRep(\G)$, let $W^{\Voddplus}$ denote the largest $\Voddplus$-subsupermodule of $W$ on which $\Voddplus$ acts trivially.
Since $\hy(\Voddplus)=\Lambdaplus$, we have $W^{\gpone}=W^{\Voddplus}$.
This space is $\Gev$-stable, and its summand $(W^{\gpone})_{\mathrm{inv}}$, on which $\Theta$ acts invertibly, belongs to $\mathscr E$.

The equivalence in part \textup{(ii)} of the following theorem is an analogue for algebraic supergroups of Zhao's equivalence for modules with typical composition factors over reduced enveloping algebras \cite[Theorem~3.11]{Zhao2011}.

\begin{theorem}\label{thm:Serre_equiv}\label{lem:unit_counit_iso}
\leavevmode
\begin{enumerate}[label=\textup{(\roman*)}]
\item Every $W\in\sRep(\G)$ has a natural decomposition
\[
W=W^{\mathrm{typ}}\oplus W^{\mathrm{atyp}},\qquad
W^{\mathrm{typ}}\cong\KacG((W^{\gpone})_{\mathrm{inv}}),
\]
where the simple composition factors of $W^{\mathrm{typ}}$ and $W^{\mathrm{atyp}}$ have typical and atypical highest weights, respectively.
The two summand functors are exact.
\item The functors
\[
\KacG:\mathscr E\longrightarrow\mathscr C^{\mathrm{typ}},
\qquad
(-)^{\gpone}:\mathscr C^{\mathrm{typ}}\longrightarrow\mathscr E
\]
are inverse equivalences of abelian categories.
\end{enumerate}
\end{theorem}

\begin{proof}
We first prove (i).
For $M\in\sRep(\Gev)$ and $W\in\sRep(\G)$, tensor induction gives the adjunction
\begin{equation}\label{eq:Kac_invariants_adjunction}
\Hom_{\G}(\KacG(M),W)\cong\Hom_{\Gev}(M,W^{\gpone}).
\end{equation}
A homomorphism on the left is sent to its restriction to $1\otimes_{\hy(\Pplus)} M$.
The inverse sends $h:M\to W^{\gpone}$ to $a\otimes_{\hy(\Pplus)}m\mapsto ah(m)$.
These maps respect the torus actions, so the hyperalgebra adjunction gives the stated adjunction for supergroups.

Put $M:=W^{\gpone}$, and let $e_W:M\to M_{\mathrm{inv}}$ be the Fitting projection.
Since $XE=0$ for every $X\in\gpone$, the operator $EF$ maps $W$ into $M$.
Define
\[
\sigma_W:W\longrightarrow M_{\mathrm{inv}};\quad
w\longmapsto\Theta_{M_{\mathrm{inv}}}^{-1}e_W(EFw).
\]
The map $\sigma_W$ is even and $\Gev$-linear by \Cref{lem:centralTheta}.
It vanishes on $\gmone W$, since $FY=0$ for every $Y\in\gmone$.
It is therefore a $\Pminus$-homomorphism when $M_{\mathrm{inv}}$ is regarded as a $\Pminus$-supermodule through $\Pminus\to\Gev$.
For $w\in M$, \eqref{eq:Theta0_PBW} gives $EFw=\Theta_Mw$.
Since both summands of the Fitting decomposition are $\Theta$-stable, we have $e_W\circ\Theta_M=\Theta_{M_{\mathrm{inv}}}\circ e_W$.
Therefore,
\[
\sigma_W(w)
=\Theta_{M_{\mathrm{inv}}}^{-1}e_W(\Theta_Mw)
=\Theta_{M_{\mathrm{inv}}}^{-1}\Theta_{M_{\mathrm{inv}}}e_W(w)
=e_W(w).
\]
In particular, $\sigma_W|_{M_{\mathrm{inv}}}=\id_{M_{\mathrm{inv}}}$.

The inclusion $M_{\mathrm{inv}}\subseteq M$ gives $i_W:\KacG(M_{\mathrm{inv}})\to W$ by \eqref{eq:Kac_invariants_adjunction}.
The map $\sigma_W$ gives $f_W:W\to\coKacG(M_{\mathrm{inv}})$ by the coinduction adjunction
\[
\Hom_{\G}(W,\coKacG(M_{\mathrm{inv}}))\cong\Hom_{\Pminus}(W,M_{\mathrm{inv}}).
\]
Both maps are even $\G$-homomorphisms, given by
\[
i_W(a\otimes_{\hy(\Pplus)} m)=am,\qquad
f_W(w)(a)=(-1)^{|a||w|}\sigma_W(aw)
\]
for homogeneous $a\in\hy(\G)$ and $w\in W$.
For $m\in M_{\mathrm{inv}}\subseteq M$, we have $f_W(m)\in\coKacG(M_{\mathrm{inv}})^{\gpone}$ and $f_W(m)(1)=m$.
By \Cref{lem:canonical_map}, $f_W\circ i_W$ and $\can_{M_{\mathrm{inv}}}$ are therefore equal on $1\otimes_{\hy(\Pplus)} M_{\mathrm{inv}}$, and hence on $\KacG(M_{\mathrm{inv}})$.
Since $\Theta_{M_{\mathrm{inv}}}$ is invertible, so is $\can_{M_{\mathrm{inv}}}$.
Thus
\[
p_W:=\can_{M_{\mathrm{inv}}}^{-1}\circ f_W:W\longrightarrow\KacG(M_{\mathrm{inv}})
\quad\text{satisfies}\quad p_W\circ i_W=\id.
\]
The Fitting projection and the adjunctions are natural in $W$, so $P_W:=i_W\circ p_W:W\to W$ is a natural even $\G$-linear idempotent.
Set
\[
W^{\mathrm{typ}}:=\operatorname{Im}(P_W),\qquad
W^{\mathrm{atyp}}:=\Ker(P_W).
\]
Then $W=W^{\mathrm{typ}}\oplus W^{\mathrm{atyp}}$ and $W^{\mathrm{typ}}\cong\KacG(M_{\mathrm{inv}})$.
Naturality means that $h\circ P_W=P_{W'}\circ h$ for every $\G$-homomorphism $h:W\to W'$.
Consequently, every short exact sequence decomposes as a direct sum of short exact sequences on the two summands.

Exactness of $\KacG$ and \Cref{cor:typical_simple_kac_realization} show that $\KacG(M_{\mathrm{inv}})$ has only typical simple composition factors.
If $W=\KacG(U)$ with $U\in\mathscr E$, the isomorphism $\KacG(U)^{\gpone}\cong U$ in \Cref{lem:canonical_map} shows that $i_W$ is an isomorphism and $P_W=\id_W$.
In particular, $P$ is the identity on every typical simple supermodule.
Since $P$ is zero on $W^{\mathrm{atyp}}$ and the two summand functors are exact, $W^{\mathrm{atyp}}$ has no typical simple composition factor.
This proves (i).

For (ii), the unit of \eqref{eq:Kac_invariants_adjunction} is an isomorphism $M\xrightarrow{\ \cong\ }\KacG(M)^{\gpone}$ for $M\in\mathscr E$, again by \Cref{lem:canonical_map}.
Hence $\KacG$ is fully faithful on $\mathscr E$.
If $W\in\mathscr C^{\mathrm{typ}}$, part (i) gives $W^{\mathrm{atyp}}=0$, so $W\cong\KacG((W^{\gpone})_{\mathrm{inv}})$.
Applying $(-)^{\gpone}$ shows that $W^{\gpone}=(W^{\gpone})_{\mathrm{inv}}\in\mathscr E$.
Thus $\KacG$ and $(-)^{\gpone}$ restrict to the stated inverse equivalences.
\end{proof}

Changing the root vectors or the admissible ordering multiplies $EF$ and $\Theta$ by the same nonzero scalar.
The space $(W^{\gpone})_{\mathrm{inv}}$ and the projection $e_W$ do not change, and the two scalar factors cancel in the map $\sigma_W$ constructed in the proof of \Cref{thm:Serre_equiv}.
Thus the decomposition in \Cref{thm:Serre_equiv}(i) is independent of these choices.

\begin{corollary}\label{cor:char_zero_simple_projective}
Assume that $\Bbbk$ has characteristic zero.
For every dominant weight $\lambda$, the simple $\G$-supermodule $L_{\G}(\lambda)$ is projective if and only if $\lambda$ is typical.
\end{corollary}

\begin{proof}
Since $\Gev$ is reductive, its finite-dimensional supermodules are completely reducible.
Thus \Cref{thm:Serre_equiv}(ii) implies that the category of typical $\G$-supermodules is semisimple.
The exact decomposition in \Cref{thm:Serre_equiv}(i) then shows that every typical simple supermodule is projective in $\sRep(\G)$.

Conversely, suppose that $L_{\G}(\lambda)$ is projective.
The $\Gev$-submodule generated by its highest-weight vector is isomorphic to $L_{\Gev}(\lambda)$ and is annihilated by $\gpone$.
By the adjunction in \eqref{eq:Kac_invariants_adjunction}, it gives a surjection
\[
\KacG(L_{\Gev}(\lambda))\longrightarrow L_{\G}(\lambda),
\]
which therefore splits.
As in the proof of \Cref{thm:Kac_simple}, every nonzero subsupermodule of $\KacG(L_{\Gev}(\lambda))$ contains $\Bbbk F\otimes L_{\Gev}(\lambda)$.
Hence this Kac module is indecomposable, so the split surjection is an isomorphism.
By \Cref{thm:Kac_simple}, $\lambda$ is typical.
\end{proof}

Sherman proved that, for a connected quasireductive supergroup over an algebraically closed field of characteristic zero, projectivity of a simple supermodule is characterized by the nonvanishing of a polynomial evaluated at its highest weight \cite[Theorem~9.11]{Sherman2025}.
For the admissible type~$\mathrm I$ supergroups considered here over such a field, and with the same choice of positive roots, his projectivity polynomial is a nonzero scalar multiple of $\Dtyp$, viewed as a polynomial on $\mathfrak h^*$.
The preceding corollary derives the resulting projectivity criterion directly from our equivalence of categories and simplicity criterion for Kac modules.


For an algebraic supergroup $\mathbf H$, write $\Ext^n_{\mathbf H}$ for the Yoneda extension groups in $\sRep(\mathbf H)$ with even morphisms, and set $\Ext^0_{\mathbf H}:=\Hom_{\mathbf H}$.
A \emph{block} is the Serre subcategory generated by an equivalence class of simple $\mathbf H$-supermodules for the relation generated by $\Ext^1_{\mathbf H}(S,S')\ne0$ or $\Ext^1_{\mathbf H}(S',S)\ne0$.

The scalar $\Dtyp(\lambda)$ is constant as $L_{\Gev}(\lambda)$ ranges over the simple modules in an ordinary $\Gev$-block.
Indeed, an extension of simple modules on which $\Theta$ has different scalar values splits into its generalized eigenspaces.
Let $\mathscr B$ be an ordinary $\Gev$-block containing a typical simple module.
Regard the objects of $\mathscr B$ as supermodules in even parity, and let $\mathscr C_{\mathscr B}$ be the Serre subcategory of $\sRep(\G)$ generated by $\KacG(L)$ for simple $L\in\mathscr B$.
Let $\Pi\mathscr C_{\mathscr B}$ denote the full subcategory of $\sRep(\G)$ whose objects are isomorphic to $\Pi V$ for $V\in\mathscr C_{\mathscr B}$.

\begin{theorem}\label{thm:typical_blocks}
The categories $\mathscr C_{\mathscr B}$ and $\Pi\mathscr C_{\mathscr B}$ are distinct blocks of $\sRep(\G)$, and
\[
\KacG:\mathscr B\xrightarrow{\ \simeq\ }\mathscr C_{\mathscr B},\qquad
\Pi\circ\KacG:\mathscr B\xrightarrow{\ \simeq\ }\Pi\mathscr C_{\mathscr B}
\]
are equivalences.
Every $\G$-block containing a typical simple $\G$-supermodule is of the form $\Pi^j\mathscr C_{\mathscr B}$ for a unique pair $(\mathscr B,j)$, where $\mathscr B$ is an ordinary $\Gev$-block containing a typical simple module and $j\in\{0,1\}$.
\end{theorem}

\begin{proof}
Let $S$ and $S'$ be typical and atypical simple $\G$-supermodules, respectively.
By \Cref{thm:Serre_equiv}(i), we have $\Ext^1_{\G}(S,S')=\Ext^1_{\G}(S',S)=0$.
Thus the blocks of $\mathscr C^{\mathrm{typ}}$ are also blocks of $\sRep(\G)$.
Since $\mathscr E$ and $\mathscr C^{\mathrm{typ}}$ are Serre subcategories of $\sRep(\Gev)$ and $\sRep(\G)$, respectively, part (ii) gives
\[
\Ext^1_{\G}\bigl(\KacG(U),\KacG(U')\bigr)
\cong\Ext^1_{\Gev}(U,U')
\qquad(U,U'\in\mathscr E).
\]
Consequently, $\KacG$ gives a bijection between their blocks.

For any $U\in\sRep(\Gev)$, the parity components $U_{\bar0}$ and $U_{\bar1}$ are $\Gev$-stable because $\Gev$ is purely even.
Thus $U=U_{\bar0}\oplus U_{\bar1}$ is a direct sum of $\Gev$-subsupermodules.
For finite-dimensional ordinary $\Gev$-modules $M,M'$ regarded as supermodules in even parity, we have
\[
\Ext^1_{\Gev}(M,\Pi M')=\Ext^1_{\Gev}(\Pi M,M')=0.
\]
Extensions between supermodules concentrated in the same parity are precisely ordinary module extensions, up to parity shift.
Since typicality is constant on ordinary blocks, every ordinary $\Gev$-block $\mathscr B$ containing a typical simple module gives two distinct blocks $\mathscr B$ and $\Pi\mathscr B$ of $\mathscr E$.
Conversely, every block of $\mathscr E$ equals $\mathscr B$ or $\Pi\mathscr B$ for some such $\mathscr B$.

By the definition of $\mathscr C_{\mathscr B}$, the equivalence sends $\mathscr B$ to $\mathscr C_{\mathscr B}$.
Since $\KacG(\Pi M)\cong\Pi\KacG(M)$ by \Cref{lem:Kac_properties}, it also sends $\Pi\mathscr B$ to $\Pi\mathscr C_{\mathscr B}$.
This proves the stated equivalences, and the bijection of blocks gives the uniqueness of $(\mathscr B,j)$.
\end{proof}

For a finite-length supermodule $U$ and a simple supermodule $S$, let $[U:S]$ denote the multiplicity of $S$ in a composition series of $U$.

\begin{corollary}\label{cor:serre_consequences}
Let $\mathscr B$ be an ordinary $\Gev$-block containing a typical simple module.
\begin{enumerate}[label=\textup{(\roman*)}]
\item If $\lambda,\mu\in\mathsf X(T)^+$ and $L_{\Gev}(\lambda),L_{\Gev}(\mu)\in\mathscr B$, then
\[
[H^0_{\G}(\lambda):L_{\G}(\mu)]=[H^0_{\Gev}(\lambda):L_{\Gev}(\mu)].
\]
\item For $M,M'\in\mathscr B$, $n\ge0$ and $j\in\{0,1\}$,
\[
\Ext^n_{\G}\bigl(\Pi^j\KacG(M),\Pi^j\KacG(M')\bigr)
\cong\Ext^n_{\Gev}(M,M').
\]
The group on the right is computed in the category of ordinary $\Gev$-modules.
If exactly one of the two modules is shifted by $\Pi$, the extension group is zero in every degree.
\end{enumerate}
\end{corollary}

\begin{proof}
The simple socle $L_{\Gev}(\lambda)$ makes $H^0_{\Gev}(\lambda)$ indecomposable, so this induced module belongs to $\mathscr B$.
By \Cref{thm:BWtensor,lem:canonical_map},
\[
H^0_{\G}(\lambda)\cong\coKacG(H^0_{\Gev}(\lambda))
\cong\KacG(H^0_{\Gev}(\lambda)).
\]
The block equivalence therefore preserves the multiplicities in (i) and gives the extension isomorphisms in (ii).
Decomposition into blocks is exact, so these extension groups are naturally isomorphic to those in the full finite-dimensional representation categories.
The two parity choices belong to distinct blocks, so all extensions between them vanish.
\end{proof}

\subsection{Frobenius kernels and projectivity}\label{subsec:frob_general}
Assume that $\Bbbk$ is perfect of odd characteristic, and fix $r\ge1$.
Write $\G_r$ and $\G_{\mathrm{ev},r}$ for the $r$-th Frobenius kernels of $\G$ and $\Gev$, respectively, and put $\Pplus_r:=(\Pplus)_r=\Pplus\cap\G_r$.
Masuoka showed that $(\G_r)_{\mathrm{ev}}=\G_{\mathrm{ev},r}$ and $\Lie(\G_r)_{\bar1}=\mathfrak g_{\bar1}$, as recalled in \cite[Proposition~3.8]{ShibataFrob}.
The PBW decomposition \eqref{eq:pbw_decomposition} therefore restricts to
\[
\Lambdaminus\otimes\hy(\G_{\mathrm{ev},r})\otimes\Lambdaplus
\xrightarrow{\ \cong\ }\hy(\G_r).
\]
Reversing the order of the odd factors gives another PBW decomposition.
In particular, $\hy(\G_r)$ is finite free on both sides over $\hy(\G_{\mathrm{ev},r})$ and over $\Lambdaplus$.
The elements $E,F$ lie in $\hy(\G_1)$, whose PBW projection takes values in $\hy(\G_{\mathrm{ev},1})$, so $\Theta\in\hy(\G_{\mathrm{ev},1})$.
It is central in $\hy(\G_{\mathrm{ev},r})$ by \Cref{lem:centralTheta}.

For $M\in\sRep(\G_{\mathrm{ev},r})$, regard $M$ as a $\Pplus_r$-supermodule through the quotient $\Pplus_r\to\G_{\mathrm{ev},r}$ and define
\[
\KacGr(M):=\hy(\G_r)\otimes_{\hy(\Pplus_r)}M.
\]
Thus $\gpone$ acts trivially on $M$.
PBW identifies $\KacGr(M)$ with $\Lambdaminus\otimes M$ as a $\G_{\mathrm{ev},r}$-supermodule, so $\KacGr$ is exact and faithful.

Put
\[
\JGr(M):=\hy(\G_r)\otimes_{\hy(\G_{\mathrm{ev},r})}M.
\]
Recall that $N=|\Donep|$.
The top exterior powers $\Lambdaplus^N=\Bbbk E$ and $\Lambdaminus^N=\Bbbk F$ have parity $N\pmod2$ and inverse $\Gev$-characters.
Hence $\Lambdaplus^N\otimes\Lambdaminus^N\cong\Bbbk$ and $\Lambdaminus^N\cong(\Lambdaplus^N)^*$.

\begin{proposition}\label{prop:iota_pi}\label{prop:composite_kernel}\label{cor:explicit_splitting}
For every $M\in\sRep(\G_{\mathrm{ev},r})$, the maps
\[
\begin{aligned}
\iota_M&:\KacGr(M)\longrightarrow\JGr(\Lambdaminus^N\otimes M);\quad
a\otimes_{\hy(\Pplus_r)} m\longmapsto aE\otimes_{\hy(\G_{\mathrm{ev},r})}(F\otimes m),\\
\pi_M&:\JGr(\Lambdaminus^N\otimes M)\longrightarrow\KacGr(M);\quad
a\otimes_{\hy(\G_{\mathrm{ev},r})}(F\otimes m)\longmapsto aF\otimes_{\hy(\Pplus_r)} m
\end{aligned}
\]
are even $\hy(\G_r)$-homomorphisms and satisfy
\[
\pi_M\circ\iota_M=\KacGr(\Theta_M).
\]
If $\Theta_M$ is invertible, then
\[
s_M:=\iota_M\circ\KacGr(\Theta_M^{-1})
\]
is a section of $\pi_M$.
For $\lambda\in\mathsf X(T)^+$ and $M=L_{\Gev}(\lambda)|_{\G_{\mathrm{ev},r}}$, $\pi_M\circ\iota_M=\Dtyp(\lambda)\id$.
Thus $s_M=\Dtyp(\lambda)^{-1}\iota_M$ when $\lambda$ is typical.
\end{proposition}

\begin{proof}
Define
\[
f_M:M\longrightarrow\JGr(\Lambdaminus^N\otimes M);\quad
m\longmapsto E\otimes_{\hy(\G_{\mathrm{ev},r})}(F\otimes m).
\]
Since $XE=0$ for $X\in\gpone$, this map is $\Lambdaplus$-linear.
For $b\in\hy(\G_{\mathrm{ev},r})$ and $z\in\hy(\G_r)$, we have
\[
bz=\sum_{(b)}(b_{(1)}\rightharpoonup z)b_{(2)}.
\]
Here we write $\Delta(b)=\sum_{(b)}b_{(1)}\otimes b_{(2)}$ for the coproduct.
We use $b\rightharpoonup z:=\sum_{(b)}b_{(1)}zS(b_{(2)})$ for the adjoint action, where $S$ is the antipode.
Let $\chi$ and $\chi'$ be the characters of $\Bbbk E$ and $\Bbbk F$, respectively, regarded as algebra homomorphisms from $\hy(\G_{\mathrm{ev},r})$ to $\Bbbk$.
They are inverse under convolution, so
\[
\begin{aligned}
bf_M(m)
&=\sum_{(b)}(b_{(1)}\rightharpoonup E)\otimes_{\hy(\G_{\mathrm{ev},r})} b_{(2)}(F\otimes m)\\
&=\sum_{(b)}\chi(b_{(1)})\chi'(b_{(2)})E\otimes_{\hy(\G_{\mathrm{ev},r})}(F\otimes b_{(3)}m)
=f_M(bm).
\end{aligned}
\]
Thus $f_M$ is $\hy(\Pplus_r)$-linear and induces $\iota_M$.

The map
\[
g_M:\Lambdaminus^N\otimes M\longrightarrow\KacGr(M);\quad
F\otimes m\longmapsto F\otimes_{\hy(\Pplus_r)} m
\]
is $\hy(\G_{\mathrm{ev},r})$-linear, since
\[
g_M(b(F\otimes m))
=\sum_{(b)}(b_{(1)}\rightharpoonup F)\otimes_{\hy(\Pplus_r)} b_{(2)}m
=bF\otimes_{\hy(\Pplus_r)} m=b\,g_M(F\otimes m).
\]
Induction extends it to the desired homomorphism $\pi_M$.
Both maps preserve parity because $|E|=|F|\equiv N\pmod2$.
The PBW expansion \eqref{eq:Theta0_PBW} gives
\[
(\pi_M\circ\iota_M)(a\otimes_{\hy(\Pplus_r)} m)=aEF\otimes_{\hy(\Pplus_r)} m=a\otimes_{\hy(\Pplus_r)}\Theta m
\]
in $\KacGr(M)$.
If $\Theta_M$ is invertible, its inverse is $\hy(\G_{\mathrm{ev},r})$-linear by centrality of $\Theta$, and the formula for $s_M$ follows.
For the restriction of $L_{\Gev}(\lambda)$, apply the scalar identity in \Cref{cor:EF_on_all_M}.
\end{proof}

If $\Theta_M$ is invertible, the section $s_M$ does not depend on the normalization of the root vectors or on the admissible ordering.
Rescaling multiplies $\iota_M$ and $\Theta$ by the same scalar, which cancels in $s_M$.
Permuting the ordering multiplies $E$ and $F$ by the same sign, so $\iota_M$ and $EF$ are unchanged.

For $\G=\GLs(m|n)$, part \textup{(ii)} of the following theorem recovers Zubkov's projectivity result for $r$-th Steinberg supermodules \cite[Lemma~10.7]{ZubkovGoodWeyl}.

\begin{theorem}\label{thm:general_projectivity}\label{prop:kernel_Kac_simple}
Let $M\in\sRep(\G_{\mathrm{ev},r})$.
\begin{enumerate}[label=\textup{(\roman*)}]
\item $\KacGr(M)$ is simple if and only if $M$ is simple and $\Theta_M$ is invertible.
\item $\KacGr(M)$ is projective if and only if $M$ is projective and $\Theta_M$ is invertible.
\item The functor $\KacGr:\sRep(\G_{\mathrm{ev},r})\to\sRep(\G_r)$ gives a bijection between the isomorphism classes of simple projective $\G_{\mathrm{ev},r}$-supermodules with invertible $\Theta$-action and those of simple projective $\G_r$-supermodules.
\end{enumerate}
For both Frobenius kernels, a finite-dimensional supermodule is projective if and only if it is injective.
\end{theorem}

\begin{proof}
Put $A:=\hy(\G_r)$ and $B:=\hy(\G_{\mathrm{ev},r})$.
For (i), we use the exterior-product argument of Zhao \cite[Lemma~3.3 and Proposition~3.5]{Zhao2011}.
Exactness and faithfulness of $\KacGr$ show that simplicity of $\KacGr(M)$ forces simplicity of $M$.
Assume now that $M$ is simple, and let $V$ be a nonzero $A$-subsupermodule of $\KacGr(M)$.
The $\Lambdaminus$-socle of $\KacGr(M)$ is $\Bbbk F\otimes M$, and its intersection with $V$ is nonzero.
The intersection is $B$-stable, and $\Bbbk F\otimes M$ is simple as a $B$-supermodule because tensoring with $\Bbbk F$ is an equivalence.
Thus $\Bbbk F\otimes M\subseteq V$.
If $\Theta_M$ is invertible, applying $E$ gives $1\otimes_{\hy(\Pplus_r)} M\subseteq V$.
These vectors generate the Kac module, proving simplicity.

Conversely, suppose that $\KacGr(M)$ is simple.
Then $M$ is a simple $B$-supermodule by the argument above.
We show that $\Theta_M$ is invertible.
Suppose otherwise.
Since $\Theta_M$ is an even $B$-endomorphism of $M$, we have $\Theta_M=0$.
Consider the nonzero submodule $A(\Bbbk F\otimes M)$.
We use the $\mathbb Z$-grading of $A$ with $\gpone$ in degree $1$, $\gmone$ in degree $-1$, and $B$ in degree $0$.
We give $\KacGr(M)$ the induced grading with $1\otimes_{\hy(\Pplus_r)}M$ in degree zero.
Consider a PBW monomial $Y_SbX_J$, where $b\in B$ and $Y_S,X_J$ are the subset products defined in the proof of \Cref{lem:centralTheta}.
The image of $\Bbbk F\otimes M$ under this monomial has degree $-N-|S|+|J|$, which can be zero only when $S=\varnothing$ and $J=\{1,\ldots,N\}$.
Its image is then zero, since $E(\Bbbk F\otimes M)=1\otimes_{\hy(\Pplus_r)}\Theta M=0$.
The submodule has zero projection onto $1\otimes_{\hy(\Pplus_r)}M$ and is therefore proper, contradicting the simplicity of $\KacGr(M)$.
Thus $\Theta_M$ is invertible.
This proves (i).

For (ii), suppose first that $M$ is $B$-projective and $\Theta_M$ is invertible.
Tensoring with $\Lambdaminus^N$ preserves projectivity, as does induction from $B$ to $A$.
Thus $\JGr(\Lambdaminus^N\otimes M)$ is projective, and \Cref{prop:iota_pi} exhibits $\KacGr(M)$ as a projective direct summand.

Conversely, let $\KacGr(M)$ be $A$-projective.
Since $A$ is free as a left $B$-module, the restriction of $\KacGr(M)$ is $B$-projective.
Each exterior-degree summand of $\Lambdaminus\otimes M$ is $B$-stable under the diagonal action.
In particular, $1\otimes_{\hy(\Pplus_r)} M$ is a direct $B$-summand, so $M$ is $B$-projective.
Since $A$ is also free as a left $\Lambdaplus$-module, the restriction of $\KacGr(M)$ to $\Lambdaplus$ is projective.
A finite-dimensional projective supermodule over this exterior algebra is a direct sum of copies of $\Lambdaplus$ and $\Pi\Lambdaplus$.
Left multiplication by $E$ has rank one on each copy, so its rank on $\KacGr(M)$ is $\dim M$.
The graded module $\KacGr(M)$ is concentrated in degrees $-N,\ldots,0$, and $E$ has degree $N$.
Hence $E$ vanishes on $\bigwedge^j(\gmone)\otimes M$ for $j<N$.
Its restriction to the degree $-N$ component is the linear map
\[
E:\KacGr(M)_{-N}\longrightarrow\KacGr(M)_0;\quad
F\otimes_{\hy(\Pplus_r)}m\longmapsto1\otimes_{\hy(\Pplus_r)}\Theta m.
\]
Under the PBW identifications $\KacGr(M)_{-N}\cong\Bbbk F\otimes M$ and $\KacGr(M)_0\cong M$, this map has the same rank as $\Theta_M$.
Hence
\[
\dim M=\rank(E|_{\KacGr(M)})=\rank(\Theta_M),
\]
which proves invertibility of $\Theta_M$.

For (iii), let $W$ be a simple projective $A$-supermodule.
The augmentation ideal of $\Lambdaplus$ is nilpotent, so $W^{\gpone}\ne0$.
The space $W^{\gpone}$ is $B$-stable, so it contains a simple $B$-subsupermodule $M'$.
Adjunction sends the inclusion to a nonzero map $\KacGr(M')\to W$, which is surjective by simplicity of $W$ and splits by its projectivity.
To show that this split surjection is an isomorphism, we prove that $\KacGr(M')$ is indecomposable.
By the PBW decomposition,
\[
\KacGr(M')/\gmone\KacGr(M')
\cong(\Lambdaminus/\gmone\Lambdaminus)\otimes M'
\cong M'
\]
as $B$-supermodules.
Suppose that $\KacGr(M')=U\oplus V$ for nonzero $A$-sub\-supermodules $U$ and $V$.
Taking the quotient by $\gmone\KacGr(M')$ gives
\[
M'\cong(U/\gmone U)\oplus(V/\gmone V)
\]
as $B$-supermodules.
Let $J$ be the augmentation ideal of $\Lambdaminus$.
If $U/\gmone U=0$, then $U=JU=\cdots=J^{N+1}U=0$, since $JU=\gmone U$ and $J^{N+1}=0$.
Thus $U/\gmone U\ne0$, and similarly $V/\gmone V\ne0$.
This contradicts the simplicity of $M'$.
Thus $\KacGr(M')$ is indecomposable, and the split surjection is an isomorphism.
Parts (i)--(ii) show that $M'$ is projective and $\Theta_{M'}$ is invertible, and also give the converse.
The quotient $\KacGr(M')/\gmone\KacGr(M')$ recovers $M'$, so the correspondence is bijective on isomorphism classes.

Finally, the bosonizations $A\#\Bbbk\mathbb Z_2$ and $B\#\Bbbk\mathbb Z_2$ are finite-dimensional ordinary Hopf algebras, and hence Frobenius algebras \cite[Example~10.6.9 and Theorem~2.1.3]{Montgomery1993}.
The categories of $A$-supermodules and $B$-supermodules are equivalent to the module categories of $A\#\Bbbk\mathbb Z_2$ and $B\#\Bbbk\mathbb Z_2$, respectively.
Therefore projectivity and injectivity coincide in both supermodule categories.
\end{proof}

Over an algebraically closed field, the case $\G=\GLs(m|n)$ and $r=1$ of the simplicity criterion in \Cref{thm:general_projectivity}(i) is given by Zheng and Shu in the setting of restricted universal enveloping superalgebras \cite[Theorem~2.2]{ZhengShu2010}.
Under the same assumption on the field, the correspondence of simple projective supermodules in \Cref{thm:general_projectivity}(iii) for $r=1$ can also be deduced from earlier results when $\G=\GLs(m|n)$ or $\SpOs(2\ell|2)$.
It follows by combining Zhao's equivalence for categories of modules with typical composition factors \cite[Theorem~3.11 and Lemma~3.10]{Zhao2011} with Yao's results on indecomposability of Kac modules and construction of simple quotients \cite[Propositions~3.3 and~3.4]{Yao2014}.

\begin{lemma}\label{lem:global_to_kernel_Kac}
For a finite-dimensional left $\hy(\Gev)$-supermodule $M$, the map
\[
\hy(\G_r)\otimes_{\hy(\Pplus_r)}M
\longrightarrow\hy(\G)\otimes_{\hy(\Pplus)}M;\quad
a\otimes_{\hy(\Pplus_r)} m\longmapsto a\otimes_{\hy(\Pplus)} m
\]
is a $\hy(\G_r)$-supermodule isomorphism.
Consequently, $\KacG(M)|_{\G_r}\cong\KacGr(M|_{\G_{\mathrm{ev},r}})$ for $M\in\sRep(\Gev)$.
\end{lemma}

\begin{proof}
The two PBW decompositions identify the map with the identity on $\Lambdaminus\otimes M$.
\end{proof}

\begin{corollary}\label{cor:H0_projectivity}
Let $\lambda\in\mathsf X(T)^+$ be typical.
Then $L_{\G}(\lambda)|_{\G_r}$ is projective if and only if $L_{\Gev}(\lambda)|_{\G_{\mathrm{ev},r}}$ is projective.
If $H^0_{\Gev}(\lambda)=L_{\Gev}(\lambda)$, the assertion also holds with $H^0_{\G}(\lambda)$ in place of $L_{\G}(\lambda)$.
\end{corollary}

\begin{proof}
By \Cref{cor:typical_simple_kac_realization,lem:global_to_kernel_Kac}, we have
\[
L_{\G}(\lambda)|_{\G_r}\cong\KacGr\bigl(L_{\Gev}(\lambda)|_{\G_{\mathrm{ev},r}}\bigr).
\]
By \Cref{cor:EF_on_all_M}, $\Theta$ acts on $L_{\Gev}(\lambda)|_{\G_{\mathrm{ev},r}}$ by the nonzero scalar $\Dtyp(\lambda)$.
Apply \Cref{thm:general_projectivity}(ii).
The last assertion follows from \Cref{thm:main_irred}.
\end{proof}

\section{Applications to types \texorpdfstring{$\mathtt A$ and $\mathtt C$}{A and C}}\label{sec:examples}
Throughout this section, $\Bbbk$ is algebraically closed of characteristic $p>2$.

\subsection{Type \texorpdfstring{$\mathtt A$}{A}}\label{subsec:typeA}
Let $\G=\GLs(m|n)$, so that $\Gev=\GLord_m\times\GLord_n$.
We write $E_{a,b}$ for the matrix unit in row $a$, column $b$.
For a dominant weight $\lambda=(\ba|\bb)=\sum_i a_i\eps_i+\sum_j b_j\del_j$ with $a_i,b_j\in\mathbb Z$, put $\beta_{ij}:=\eps_i-\del_j$ and choose
\[
X_{\beta_{ij}}:=E_{i,m+j},\qquad
Y_{\beta_{ij}}:=E_{m+j,i}.
\]
Then $H_{\beta_{ij}}=E_{i,i}+E_{m+j,m+j}$, and the typicality scalar is
\[
\mathscr D_{\GLs(m|n)}(\ba|\bb)
=\prod_{i=1}^m\prod_{j=1}^n(a_i+b_j+m+1-i-j).
\]
Thus \Cref{thm:main_irred} recovers Marko's criterion that $H^0_{\GLs(m|n)}(\lambda)$ is simple if and only if $H^0_{\Gev}(\lambda)$ is simple and none of the integers $a_i+b_j+m+1-i-j$ is divisible by $p$ \cite{Marko2018}.

C.~Zhang proved that the same nonvanishing condition characterizes simplicity of Kac modules for $\mathfrak{gl}(m|n)$ in odd characteristic, induced from simple highest-weight modules for its even Lie algebra \cite{ZhangKac2014}.
For the algebraic supergroup, \Cref{thm:main_irred} gives the criterion for $H^0_{\GLs(m|n)}(\lambda)$, including the required simplicity of $H^0_{\Gev}(\lambda)$.
The block theorem of Marko and Zubkov \cite{MarkoZubkovGLBlocks} describes the sets of highest weights by chains of dominant weights in which each step stays in an ordinary $\Gev$-block or has the form $\lambda\leftrightarrow\lambda\pm\beta_{ij}$ with $a_i+b_j+m+1-i-j=0$ in $\Bbbk$.
Since $\Dtyp$ is constant on each ordinary $\Gev$-block, every even step from a typical weight preserves typicality, and no odd step can occur.
The highest-weight set of a typical $\G$-block is therefore that of an ordinary $\Gev$-block.
With even morphisms and the two parities distinguished, \Cref{thm:typical_blocks} also gives an equivalence between each typical $\G$-block and the corresponding ordinary $\Gev$-block.

Let $\mathbf H$ be $\SLs(m|n)$ with $m\ne n$, $\PGLs(m|n)$, or $\PSLs(m|n)$ with $m\ne n$, and let $\overline\lambda$ be a dominant character of its maximal torus.
For $\SLs$, choose a weight $\lambda=(\ba|\bb)$ of $\GLs(m|n)$ that restricts to $\overline\lambda$.
For $\PGLs$, set $\lambda:=\overline\lambda\circ q|_T$, where $q:\GLs(m|n)\to\PGLs(m|n)$ is the quotient map.
For $\PSLs$, use the isomorphism with $\PGLs(m|n)$ in the proof of \Cref{lem:PGL_over_Z} and the same formula for $\lambda$.

\begin{proposition}\label{prop:typeA_central_forms}
The induced supermodule $H^0_{\mathbf H}(\overline\lambda)$ is simple if and only if $H^0_{\mathbf H_{\mathrm{ev}}}(\overline\lambda)$ is simple and
\[
a_i+b_j+m+1-i-j\ne0\quad\text{in }\Bbbk
\qquad(1\le i\le m,\ 1\le j\le n).
\]
The condition does not depend on the choice of the lift $\lambda$ for $\SLs(m|n)$.
\end{proposition}

\begin{proof}
Two lifts for $\SLs(m|n)$ differ by an integral multiple of the Berezinian character.
As a weight, the Berezinian character is $\sum_i\eps_i-\sum_j\del_j$, which vanishes on every $H_{\beta_{ij}}=E_{i,i}+E_{m+j,m+j}$.
Thus each factor is independent of the lift.
For $\PGLs(m|n)$, the character $\lambda$ satisfies $\sum_i a_i+\sum_j b_j=0$.
The quotient map $\mathfrak{gl}(m|n)\to\mathfrak{pgl}(m|n)$ sends $H_{\beta_{ij}}$ to the bracket of the chosen odd root images, so it gives the same factor.
The assertion now follows from \Cref{thm:main_irred} and, for $\PSLs$, from \Cref{prop:central_forms}.
\end{proof}

In particular, a character of the torus of $\PGLs(1|1)$ satisfies $a_1+b_1=0$.
Its typicality scalar is therefore zero, so this supergroup has no typical dominant weight and no simple induced supermodule.

\subsection{Type \texorpdfstring{$\mathtt C$}{C}}\label{subsec:typeC}
Let $\ell\ge1$ and $\G=\SpOs(2\ell|2)$, with $\Gev=\Spord_{2\ell}\times\SOord_2$.

Write $\eps$ for the standard character of $\SOord_2\cong\Gm$ and $\del_i$ for the standard weights of $\Spord_{2\ell}$.
The positive roots are
\[
\begin{aligned}
\Dzerop&=\{\del_i-\del_j,\ \del_i+\del_j\mid i<j\}
\sqcup\{2\del_i\mid1\le i\le\ell\},\\
\Donep&=\{\eps-\del_i,\ \eps+\del_i\mid1\le i\le\ell\}.
\end{aligned}
\]
For $a,b_i\in\mathbb Z$, write $(\bb|a):=\sum_i b_i\del_i+a\eps$.
The character lattices from \Cref{prop:typeC_adjoint_setup} give
\[
\begin{aligned}
\mathsf X(T)^+&=\{(\bb|a)\mid a,b_i\in\mathbb Z,\ b_1\ge\cdots\ge b_\ell\ge0\},\\
\mathsf X(\Tq)^+&=\{(\bb|a)\in\mathsf X(T)^+\mid a+\textstyle\sum_i b_i\equiv0\pmod2\}.
\end{aligned}
\]

\subsubsection{Typicality and characters}\label{subsubsec:typeC_typicality}

Use the homogeneous basis of \Cref{subsubsec:setup_typeC}.
Put $H_\eps:=E_{+,+}-E_{-,-}$ and $H_i:=E_{i,i}-E_{-i,-i}$.
Choose the odd root vectors
\[
\begin{aligned}
X_{\eps-\del_i}&:=E_{+,i}-E_{-i,-}, &Y_{\eps-\del_i}&:=E_{i,+}+E_{-,-i},\\
X_{\eps+\del_i}&:=-E_{i,-}-E_{+,-i}, &Y_{\eps+\del_i}&:=E_{-,i}-E_{-i,+}.
\end{aligned}
\]
Matrix multiplication gives
\[
H_{\eps-\del_i}=H_\eps+H_i,\qquad
H_{\eps+\del_i}=H_\eps-H_i.
\]

The positive roots give
\begin{equation}\label{eq:typeC_rho}
\rho_{\bar0}=\sum_{i=1}^{\ell}(\ell-i+1)\del_i,\qquad
\rho=\rho_{\bar0}-\ell\eps.
\end{equation}

\begin{proposition}\label{prop:typeC_typicality}
For $\lambda=(\bb|a)\in\mathsf X(T)^+$, the typicality scalar is
\[
\Dtyp(\bb|a)
=\prod_{i=1}^{\ell}(a+b_i-i+1)(a-b_i-2\ell+i-1).
\]
In particular, $\lambda$ is typical if and only if both factors are nonzero in $\Bbbk$ for every $i$.
\end{proposition}

\begin{proof}
Using \eqref{eq:typeC_rho} and the chosen brackets, we obtain
\[
\begin{aligned}
(\lambda+\rho)(H_{\eps-\del_i})&=a+b_i-i+1,\\
(\lambda+\rho)(H_{\eps+\del_i})&=a-b_i-2\ell+i-1.
\end{aligned}
\]
Taking the product over $i$ gives the stated formula.
\end{proof}

For $a\in\mathbb Z$, let $\Bbbk_a$ be the one-dimensional $\SOord_2$-module of weight $a$.
Write $\boxtimes$ for the external tensor product.
For $b_1\ge\cdots\ge b_\ell\ge0$, write $L_{\Spord_{2\ell}}(\bb)$ and $H^0_{\Spord_{2\ell}}(\bb)$ for the simple and induced modules of highest weight $\sum_i b_i\del_i$.
We have
\[
L_{\Gev}(\bb|a)\cong L_{\Spord_{2\ell}}(\bb)\boxtimes\Bbbk_a,\qquad
H^0_{\Gev}(\bb|a)\cong H^0_{\Spord_{2\ell}}(\bb)\boxtimes\Bbbk_a.
\]
Since
\[
\ch(\Lambdaminus)=\prod_{i=1}^{\ell}(1+e^{-\eps+\del_i})(1+e^{-\eps-\del_i}),
\]
\Cref{lem:Kac_properties,thm:BWtensor} give
\[
\ch\bigl(H^0_{\G}(\bb|a)\bigr)
=e^{a\eps}\ch\bigl(H^0_{\Spord_{2\ell}}(\bb)\bigr)\ch(\Lambdaminus).
\]
If $(\bb|a)$ is typical, \Cref{cor:typical_simple_kac_realization} also gives
\[
\ch\bigl(L_{\G}(\bb|a)\bigr)
=e^{a\eps}\ch\bigl(L_{\Spord_{2\ell}}(\bb)\bigr)\ch(\Lambdaminus).
\]
In particular, $H^0_{\G}(\bb|a)$ has dimension $4^\ell\dim H^0_{\Spord_{2\ell}}(\bb)$.
For typical $(\bb|a)$, the module $L_{\G}(\bb|a)$ has dimension $4^\ell\dim L_{\Spord_{2\ell}}(\bb)$.

\begin{example}\label{ex:OSP24_wall}
Let $\ell=2$, $p=3$, and $\bb=(1,0)$.
The formula becomes
\[
\mathscr D_{\SpOs(4|2)}(1,0|a)=(a+1)(a-5)(a-1)(a-3).
\]
Every residue class of $a$ modulo $3$ makes at least one factor zero.
Thus none of the weights $(1,0|a)$ is typical.
Although $H^0_{\Spord_4}(1,0)$ is the simple natural four-dimensional module, $H^0_{\SpOs(4|2)}(1,0|a)$ is reducible for every $a$ by \Cref{thm:main_irred}.
\end{example}

Write $q:\G\to\mathbf{PSpO}(2\ell|2)$ for the quotient map.
Its kernel $D_\ell\cong\mu_2$ is finite \'etale, so $\Lie(q)$ is an isomorphism and carries the chosen root vectors and their brackets to those of the quotient.
Consequently,
\[
\mathscr D_{\mathbf{PSpO}(2\ell|2)}(\lambda)
=\Dtyp(\lambda)
\qquad\bigl(\lambda\in\mathsf X(\Tq)^+\bigr),
\]
and the preceding character and dimension formulas hold on this character lattice as well.

\subsubsection{Typical blocks and extensions}\label{subsec:typeC_blocks}

Let $\WeylC$ be the Weyl group of $\Spord_{2\ell}$.
It acts by signed permutations of the $\del_i$ and fixes $\eps$.
For $w\in\WeylC$ and $\eta\in\mathsf X(T)$, put $w\mathbin{\boldsymbol\cdot}\eta:=w(\eta+\rho_{\bar0})-\rho_{\bar0}$.
Write $\Psi_{\bar0}:=\{\del_i-\del_{i+1}\mid1\le i<\ell\}\sqcup\{2\del_\ell\}$ for the simple even roots.
For a dominant weight $\eta$, define
\[
\begin{aligned}
d(\eta)&:=\max\{d\in\mathbb Z_{\ge0}\mid\langle\eta+\rho_{\bar0},\alpha^\vee\rangle\in p^d\mathbb Z
\text{ for all }\alpha\in\Psi_{\bar0}\},\\
\mathcal C(\eta)&:=\WeylC\mathbin{\boldsymbol\cdot}\eta+p^{d(\eta)+1}\mathbb Z\Dzero.
\end{aligned}
\]
The pairings $\langle\eta+\rho_{\bar0},\alpha^\vee\rangle$ are positive integers, so the maximum defining $d(\eta)$ exists.

\begin{proposition}\label{prop:typeC_even_block_formula}
For $\eta\in\mathsf X(T)^+$, the highest weights in the ordinary $\Gev$-block of $L_{\Gev}(\eta)$ are exactly $\mathcal C(\eta)\cap\mathsf X(T)^+$.
\end{proposition}

\begin{proof}
The $\SOord_2$-character is constant on a block because $\SOord_2$ is diagonalizable.
For $\Spord_{2\ell}$, the fundamental weights form a basis of the weight lattice.
Thus $d(\eta)$ is the largest $d$ for which the $\Spord_{2\ell}$-coordinate of $\eta+\rho_{\bar0}$ belongs to $p^d$ times that lattice.
Donkin's block formula \cite[Theorem~5.8]{Donkin1980} (see also \cite[Theorem~1]{DeVisscher2008}), applied to $\Spord_{2\ell}$ with this value of $d(\eta)$, gives the asserted weight set.
\end{proof}

\begin{theorem}\label{thm:typeC_blocks}
Let $\lambda\in\mathsf X(T)^+$ be typical.
\begin{enumerate}[label=\textup{(\roman*)},beginpenalty=10000]
\item For $j\in\{0,1\}$, the simple objects in the $\G$-block of $\Pi^jL_{\G}(\lambda)$ are exactly
\[
\Pi^jL_{\G}(\mu),\qquad\mu\in\mathcal C(\lambda)\cap\mathsf X(T)^+.
\]
The functor $\Pi^j\circ\KacG$ is an equivalence from the ordinary $\Gev$-block of $L_{\Gev}(\lambda)$ to this block.
\item For $\mu,\nu\in\mathcal C(\lambda)\cap\mathsf X(T)^+$, $n\ge0$, and $j\in\{0,1\}$,
\[
\Ext^n_{\G}\bigl(\Pi^jL_{\G}(\mu),\Pi^jL_{\G}(\nu)\bigr)
\cong\Ext^n_{\Gev}(L_{\Gev}(\mu),L_{\Gev}(\nu)).
\]
The group on the right is taken in the category of ordinary $\Gev$-modules.
If $i,j\in\{0,1\}$ and $i\ne j$, then
\[
\Ext^n_{\G}\bigl(\Pi^iL_{\G}(\mu),\Pi^jL_{\G}(\nu)\bigr)=0
\]
for all $n\ge0$.
\end{enumerate}
For a typical $\lambda\in\mathsf X(\Tq)^+$, both assertions hold for $\mathbf{PSpO}(2\ell|2)$, with $\mathsf X(\Tq)^+$ and $\mathbf{PSpO}(2\ell|2)_{\mathrm{ev}}$ in place of $\mathsf X(T)^+$ and $\Gev$.
\end{theorem}

\begin{proof}
For $\G=\SpOs(2\ell|2)$, apply \Cref{thm:typical_blocks,cor:serre_consequences,prop:typeC_even_block_formula} together with the isomorphism
$L_{\G}(\mu)\cong\KacG(L_{\Gev}(\mu))$ of \Cref{cor:typical_simple_kac_realization}.
For the central quotient, every $\G$-supermodule decomposes into its two $D_\ell\cong\mu_2$-character spaces.
These are $\G$-stable because $D_\ell$ is central, and this decomposition is exact.
Via the quotient map $\G\to\mathbf{PSpO}(2\ell|2)$, we identify $\sRep(\mathbf{PSpO}(2\ell|2))$ with the full subcategory of $\sRep(\G)$ consisting of supermodules on which $D_\ell$ acts trivially.
This identification preserves blocks and all extension groups.
The same argument applies to the even groups.
Every even root is trivial on $D_\ell$, and
$w\mathbin{\boldsymbol\cdot}\lambda-\lambda\in\mathbb Z\Dzero$ for $w\in\WeylC$.
It follows that $\mathcal C(\lambda)\subseteq\mathsf X(\Tq)$ whenever $\lambda\in\mathsf X(\Tq)$.
The entire even block therefore descends to $\mathbf{PSpO}(2\ell|2)_{\mathrm{ev}}$.
\end{proof}

Marko and Zubkov studied linkage for general ortho-symplectic supergroups \cite[Theorem~9.2]{MarkoZubkovOSP}.
Their general linkage theorem gives a necessary condition for two highest weights to belong to the same block.
For typical $\SpOs(2\ell|2)$-blocks, \Cref{thm:typeC_blocks} gives a necessary and sufficient condition for membership in the same block, with parity fixed.
We obtain this condition by combining our block equivalence with Donkin's description of the blocks of $\Spord_{2\ell}$ \cite{Donkin1980}.

We illustrate the extension isomorphism with $\ell=1$, so that $\G=\SpOs(2|2)$ and $\Gev\cong\SLord_2\times\Gm$.
Write $L_{\SLord_2}(b)$ and $H^0_{\SLord_2}(b)$ for the simple and induced modules of highest weight $b\ge0$.

\begin{proposition}\label{thm:rankone_extension}
Let $a\in\mathbb Z$ with $a\not\equiv0,2\pmod p$.
There is a nonsplit exact sequence
\[
0\longrightarrow L_{\G}(p|a)\longrightarrow H^0_{\G}(p|a)
\longrightarrow L_{\G}(p-2|a)\longrightarrow0,
\]
and $\dim\Ext^1_{\G}(L_{\G}(p-2|a),L_{\G}(p|a))=1$.
\end{proposition}

\begin{proof}
By duality, \cite[equation~(3)]{Parker2007} gives the nonsplit $\SLord_2$-sequence
\[
0\longrightarrow L_{\SLord_2}(p)\longrightarrow H^0_{\SLord_2}(p)
\longrightarrow L_{\SLord_2}(p-2)\longrightarrow0.
\]
Apply $(-)\boxtimes\Bbbk_a$ and then the exact functor $\coKacG$.
Since $\Dtyp(p|a)=\Dtyp(p-2|a)=a(a-2)\ne0$ in $\Bbbk$, both weights are typical.
Hence \Cref{cor:typical_simple_kac_realization,thm:BWtensor} identify the resulting sequence with the one in the statement.
The two even simple modules belong to one typical block, on which $\coKacG\cong\KacG$ by \Cref{cor:Kac_dualKac_block}.
If the resulting sequence split, applying $(-)^{\gpone}$ would split the even sequence.
By \Cref{thm:typeC_blocks} and diagonalizability of $\Gm$, the extension group is isomorphic to
$\Ext^1_{\SLord_2}(L_{\SLord_2}(p-2),L_{\SLord_2}(p))$.
Since $p-2<p$, \cite[Part~II, Proposition~4.14]{Jantzen2003} gives
$\Ext^1_{\SLord_2}(L_{\SLord_2}(p-2),H^0_{\SLord_2}(p))=0$.
The long exact sequence associated with the above short exact sequence therefore gives
\[
\Ext^1_{\SLord_2}(L_{\SLord_2}(p-2),L_{\SLord_2}(p))
\cong\Hom_{\SLord_2}(L_{\SLord_2}(p-2),L_{\SLord_2}(p-2))\cong\Bbbk.
\]
Thus the extension group in the statement is one-dimensional.
\end{proof}

\subsubsection{Simple projective modules over Frobenius kernels}\label{subsec:typeC_steinberg}

Let $\G=\SpOs(2\ell|2)$ with $\ell\ge1$.
Fix $r\ge1$, and put $q:=p^r$.
Let
\[
\St_q:=L_{\Spord_{2\ell}}((q-1)\rho_{\bar0})=H^0_{\Spord_{2\ell}}((q-1)\rho_{\bar0})
\]
be the $r$-th Steinberg module of $\Spord_{2\ell}$.
For $a\in\mathbb Z$, define
\[
\lambda_a:=(q-1)\rho_{\bar0}+a\eps,\qquad
M_a:=\St_q\boxtimes\Bbbk_a,\qquad
M_{a,r}:=M_a|_{\G_{\mathrm{ev},r}}.
\]
The isomorphism class of $M_{a,r}$ depends only on $a$ modulo $q$.

\begin{theorem}\label{thm:typeC_steinberg}
\begin{enumerate}[label=\textup{(\roman*)}]
\item Up to parity shift, the simple projective $\G_r$-supermodules are exactly
\[
\KacGr(M_{a,r}),\qquad
a\in\mathbb Z/q\mathbb Z,\quad a\not\equiv\ell\pmod p.
\]
They are pairwise nonisomorphic and have dimension $4^\ell q^{\ell^2}$.
There are $q-q/p$ classes up to parity shift.
All these modules are injective.
\item For every integer $a$ with $a\not\equiv\ell\pmod p$,
\[
H^0_{\G}(\lambda_a)=L_{\G}(\lambda_a),\qquad
H^0_{\G}(\lambda_a)|_{\G_r}\cong\KacGr(M_{a,r}),
\]
and
\[
\ch\bigl(H^0_{\G}(\lambda_a)\bigr)
=e^{a\eps}\ch(\St_q)\ch(\Lambdaminus).
\]
\item The central quotient induces an isomorphism
$\G_r\cong\mathbf{PSpO}(2\ell|2)_r$,
so the classification in \textup{(i)} also holds for the central quotient.
For every even integer $a$ with $a\not\equiv\ell\pmod p$, $H^0_{\G}(\lambda_a)$ descends to
\[
H^0_{\mathbf{PSpO}(2\ell|2)}(\lambda_a)=L_{\mathbf{PSpO}(2\ell|2)}(\lambda_a).
\]
Every simple projective $\mathbf{PSpO}(2\ell|2)_r$-supermodule is the restriction of one of these modules, up to parity shift.
\end{enumerate}
\end{theorem}

\begin{proof}
For (i), set $K:=\Spord_{2\ell}$ for simplicity.
Every ordinary simple $K_r$-module is the restriction of $L_K(\eta)$ for a unique dominant weight $\eta$ of $K$ with
$0\le\langle\eta,\alpha^\vee\rangle<q$ for $\alpha\in\Psi_{\bar0}$ \cite[Part~II, Proposition~3.15 and Remark~2]{Jantzen2003}.
Let $U_K^+$ be the unipotent radical of the positive Borel subgroup of $K$.
If $L_K(\eta)|_{K_r}$ is projective, its restriction to $(U_K^+)_r$ is projective, since $\hy(K_r)$ is free over $\hy((U_K^+)_r)$ by PBW.
The algebra $\hy((U_K^+)_r)$ is local of dimension $q^{\ell^2}$.
The restriction is therefore free, so $q^{\ell^2}$ divides $\dim L_K(\eta)$.
On the other hand,
\[
\dim L_K(\eta)
\le\dim H^0_K(\eta)
=\prod_{\alpha\in\Dzerop}
\frac{\langle\eta+\rho_{\bar0},\alpha^\vee\rangle}
{\langle\rho_{\bar0},\alpha^\vee\rangle}
\le q^{\ell^2}.
\]
For the last inequality, express each positive coroot as a nonnegative integral combination of simple coroots.
Every simple-coroot pairing with $\eta+\rho_{\bar0}$ is at most $q$.
If $\eta\ne(q-1)\rho_{\bar0}$, at least one simple-coroot pairing is strictly less than $q$. Its factor occurs in the displayed product, so the upper bound is strict.
The divisibility thus forces $\eta=(q-1)\rho_{\bar0}$.
Conversely, $\St_q|_{K_r}$ is simple and projective \cite[Part~II, Proposition~10.2]{Jantzen2003}.

Since $(\SOord_2)_r=\mu_q$ is diagonalizable, the ordinary simple projective $\G_{\mathrm{ev},r}$-modules are exactly the $M_{a,r}$ for $a\in\mathbb Z/q\mathbb Z$.
In \Cref{prop:typeC_typicality}, substituting $b_i=(q-1)(\ell-i+1)$ makes all $2\ell$ factors equal to $a-\ell$ in $\Bbbk$.
By \Cref{thm:odd_product,cor:EF_on_all_M}, $\Theta$ acts on $M_{a,r}$ by $(a-\ell)^{2\ell}$.
Now \Cref{thm:general_projectivity} gives exactly the modules in (i).
We have $\KacGr(M_{a,r})/\gmone\KacGr(M_{a,r})\cong M_{a,r}$ as $\G_{\mathrm{ev},r}$-supermodules.
Taking this quotient commutes with parity shifts, so different values of $a$ modulo $q$ and distinct parities give nonisomorphic supermodules.
There are $q/p$ excluded residues, and PBW gives dimension $2^{2\ell}\dim\St_q=4^\ell q^{\ell^2}$.
Injectivity follows from the last assertion of \Cref{thm:general_projectivity}.

For (ii), the even induced module $H^0_{\Gev}(\lambda_a)=M_a$ is simple, and $a\not\equiv\ell\pmod p$ makes $\lambda_a$ typical.
By \Cref{thm:main_irred}, $H^0_{\G}(\lambda_a)$ is simple.
Its restriction is $\KacGr(M_{a,r})$ by \Cref{cor:typical_simple_kac_realization,lem:global_to_kernel_Kac}, and its character is given by \eqref{eq:typical_simple_character}.

For (iii), the central subgroup $D_\ell\cong\mu_2$ is finite \'etale of order prime to $p$.
Its $r$-th Frobenius morphism is an isomorphism, so the Frobenius diagram for
$1\to D_\ell\to\G\to\mathbf{PSpO}(2\ell|2)\to1$ gives the asserted isomorphism of kernels.
Since $q-1$ is even, the character condition in \Cref{prop:typeC_adjoint_setup} says that $\lambda_a$ descends if and only if $a$ is even.
For such $a$, we identify $H^0_{\mathbf{PSpO}(2\ell|2)}(\lambda_a)$ with $H^0_{\G}(\lambda_a)$ via the quotient map $\G\to\mathbf{PSpO}(2\ell|2)$.
Every residue class modulo the odd integer $q$ has an even lift.
Hence these descended modules give every class in (i), up to parity shift.
\end{proof}

We conclude by comparing our type~$\mathtt C$ results with those of Porter and Young.
They study finite-dimensional weight supermodules for the restricted unrolled quantum supergroup of $\mathfrak{osp}(2|2\ell)$ over $\mathbb C$.
They assume that the root of unity has odd order at least three not dividing $\ell$.
Their simplicity criterion for Kac modules uses the quantum numbers corresponding to the factors in our type~$\mathtt C$ typicality scalar \cite[Proposition~3.8]{PorterYoung2026}.
The product formula used in their proof was obtained by R.~B. Zhang \cite{Zhang1993}.

Porter and Young also construct a simple projective Kac module induced from an even module of Steinberg highest weight \cite[Corollary~3.24]{PorterYoung2026}.
For the Frobenius kernels of $\SpOs(2\ell|2)$, \Cref{thm:typeC_steinberg} shows that every simple projective supermodule is obtained, up to parity shift, by Kac induction from an even module whose symplectic factor is a Steinberg module.

\raggedbottom


\begin{thebibliography}{Par07}

\bibitem[CM21]{ChenMazorchuk2021}
C.-W.~Chen and V.~Mazorchuk,
\emph{Simple supermodules over Lie superalgebras},
Trans. Amer. Math. Soc. \textbf{374} (2021), no.~2, 899--921.

\bibitem[DV08]{DeVisscher2008}
M.~De Visscher,
\emph{On the blocks of semisimple algebraic groups and associated generalized Schur algebras},
J. Algebra \textbf{319} (2008), no.~3, 952--965.
\href{https://doi.org/10.1016/j.jalgebra.2007.11.015}{doi:10.1016/j.jalgebra.2007.11.015}.

\bibitem[Don80]{Donkin1980}
S.~Donkin,
\emph{The blocks of a semisimple algebraic group},
J. Algebra \textbf{67} (1980), no.~1, 36--53.
\href{https://doi.org/10.1016/0021-8693(80)90305-1}{doi:10.1016/0021-8693(80)90305-1}.

\bibitem[FG12]{FioresiGavarini2012}
R.~Fioresi and F.~Gavarini,
\emph{Chevalley supergroups},
Mem. Amer. Math. Soc. \textbf{215} (2012), no.~1014.

\bibitem[Gor02]{Gorelik2002}
M.~Gorelik,
\emph{Strongly typical representations of the basic classical Lie superalgebras},
J. Amer. Math. Soc. \textbf{15} (2002), no.~1, 167--184.

\bibitem[Hay25]{Hayashi2025}
T.~Hayashi,
\emph{Classification of irreducible representations of affine group superschemes and the division superalgebras of their endomorphisms},
Math. Z. \textbf{309} (2025), Paper No.~33.

\bibitem[Jan03]{Jantzen2003}
J.~C.~Jantzen,
\emph{Representations of Algebraic Groups},
2nd ed., Math. Surveys Monogr. \textbf{107},
American Mathematical Society, Providence, RI, 2003.

\bibitem[Kac77]{Kac1977}
V.~G. Kac,
\emph{Characters of typical representations of classical Lie superalgebras},
Comm. Algebra \textbf{5} (1977), 889--897.

\bibitem[Kac78]{Kac1978}
V.~G. Kac,
\emph{Representations of classical Lie superalgebras},
in Differential Geometrical Methods in Mathematical Physics II,
Lecture Notes in Math. \textbf{676}, Springer, 1978, 597--626.

\bibitem[LS23]{LiShu2023}
Y.-Y.~Li and B.~Shu,
\emph{Jantzen filtration of Weyl modules for general linear supergroups},
Forum Math. \textbf{35} (2023), no.~5, 1435--1468.

\bibitem[Mar18]{Marko2018}
F.~Marko,
\emph{Irreducibility of induced supermodules for general linear supergroups},
J. Algebra \textbf{494} (2018), 92--110.

\bibitem[MZ18a]{MarkoZubkovGLBlocks}
F.~Marko and A.~N. Zubkov,
\emph{Blocks for the general linear supergroup $GL(m|n)$},
Transform. Groups \textbf{23} (2018), 185--215.

\bibitem[MZ18b]{MarkoZubkovOSP}
F.~Marko and A.~N. Zubkov,
\emph{Linkage principle for ortho-symplectic supergroups},
J. Algebra \textbf{493} (2018), 444--482.

\bibitem[Mas12]{Masuoka2012}
A.~Masuoka,
\emph{Harish-Chandra pairs for algebraic affine supergroup schemes over an arbitrary field},
Transform. Groups \textbf{17} (2012), 1085--1121.

\bibitem[MS17]{MasuokaShibata2017}
A.~Masuoka and T.~Shibata,
\emph{Algebraic supergroups and Harish-Chandra pairs over a commutative ring},
Trans. Amer. Math. Soc. \textbf{369} (2017), 3443--3481.

\bibitem[MZ11]{MasuokaZubkov2011}
A.~Masuoka and A.~N. Zubkov,
\emph{Quotient sheaves of algebraic supergroups are superschemes},
J. Algebra \textbf{348} (2011), 135--170.

\bibitem[Mon93]{Montgomery1993}
S.~Montgomery,
\emph{Hopf algebras and their actions on rings},
CBMS Regional Conference Series in Mathematics \textbf{82}, American Mathematical Society, 1993.

\bibitem[Par07]{Parker2007}
A.~E. Parker,
\emph{Higher extensions between modules for $\SLord_2$},
Adv. Math. \textbf{209} (2007), no.~1, 381--405.
\href{https://doi.org/10.1016/j.aim.2006.05.015}{doi:10.1016/j.aim.2006.05.015}.

\bibitem[PS92]{PenkovSerganova1992}
I.~Penkov and V.~Serganova,
\emph{Representations of classical Lie superalgebras of type I},
Indag. Math. (N.S.) \textbf{3} (1992), no.~4, 419--466.

\bibitem[PY26]{PorterYoung2026}
S.~Porter and M.~B. Young,
\emph{Relative modular categories from $\mathfrak{osp}(2|2n)$},
2026, \href{https://arxiv.org/abs/2608.06333v1}{arXiv:\allowbreak2608.06333v1}.

\bibitem[Ser14]{Serganova2014}
V.~Serganova,
\emph{Finite dimensional representations of algebraic supergroups},
in Proceedings of the International Congress of Mathematicians--Seoul 2014,
Vol.~I, Kyung Moon Sa, Seoul, 2014, 603--632.

\bibitem[She25]{Sherman2025}
A.~Sherman,
\emph{Ghost distributions on supersymmetric spaces I: Koszul induced superspaces, branching, and the full ghost centre},
Transform. Groups \textbf{30} (2025), 869--931.

\bibitem[Shi20]{Shibata2020}
T.~Shibata,
\emph{Borel--Weil theorem for algebraic supergroups},
J. Algebra \textbf{547} (2020), 179--219.

\bibitem[Shi24]{ShibataFrob}
T.~Shibata,
\emph{Frobenius kernels of algebraic supergroups and Steinberg's tensor product theorem},
Algebr. Represent. Theory \textbf{27} (2024), 927--959.

\bibitem[SW08]{ShuWang2008}
B.~Shu and W.~Wang,
\emph{Modular representations of the ortho-symplectic supergroups},
Proc. Lond. Math. Soc. (3) \textbf{96} (2008), 251--271.

\bibitem[Yao14]{Yao2014}
Y.-F.~Yao,
\emph{On representations of restricted Lie superalgebras},
Czechoslovak Math. J. \textbf{64} (2014), no.~3, 845--856.
DOI: \href{https://doi.org/10.1007/s10587-014-0137-2}{10.1007/s10587-014-0137-2}.

\bibitem[Zha14]{ZhangKac2014}
C.~Zhang,
\emph{On the simplicity of Kac modules for the restricted Lie superalgebra $\mathfrak{gl}(m,n)$},
2014, revised 2016, \href{https://arxiv.org/abs/1404.7342v2}{arXiv:1404.7342v2}.

\bibitem[Zha93]{Zhang1993}
R.~B. Zhang,
\emph{Finite-dimensional representations of $U_q(C(n+1))$ at arbitrary $q$},
J. Phys. A: Math. Gen. \textbf{26} (1993), no.~23, 7041--7059.

\bibitem[Zha11]{Zhao2011}
L.~Zhao,
\emph{Typical blocks of Lie superalgebras in prime characteristic},
Comm. Algebra \textbf{39} (2011), 534--547.

\bibitem[ZS10]{ZhengShu2010}
L.~Zheng and B.~Shu,
\emph{On representations of $gl(m|n)$ and infinitesimal subgroups of $GL(m|n)$},
Chinese Ann. Math. Ser. A \textbf{31} (2010), no.~2, 129--142 (in Chinese).

\bibitem[Zub15]{ZubkovGoodWeyl}
A.~N. Zubkov,
\emph{$GL(m|n)$-supermodules with good and Weyl filtrations},
J. Pure Appl. Algebra \textbf{219} (2015), 5259--5279.

\end{thebibliography}
\end{document}